\documentclass[11pt,reqno]{amsart}

\usepackage[a4paper,margin=1.25in]{geometry}
\usepackage[T1]{fontenc}
\usepackage[utf8]{inputenc}
\usepackage{lmodern}
\usepackage{microtype}
\usepackage{amsmath,amssymb,amsthm,mathtools}

\usepackage[inline]{enumitem}
\setlist[enumerate,1]{label={\upshape (\roman*)}}
\usepackage{aliascnt}
\usepackage[colorlinks=true,linkcolor=blue,citecolor=blue,urlcolor=blue]{hyperref}
\usepackage[nameinlink,capitalize]{cleveref}
\usepackage{lineno}

\numberwithin{equation}{section}
\allowdisplaybreaks

\theoremstyle{plain}
\newtheorem{theorem}{Theorem}[section]

\newaliascnt{lemma}{theorem}
\newtheorem{lemma}[lemma]{Lemma}
\aliascntresetthe{lemma}

\newaliascnt{conjecture}{theorem}

\aliascntresetthe{conjecture}

\newaliascnt{corollary}{theorem}
\newtheorem{corollary}[corollary]{Corollary}
\aliascntresetthe{corollary}

\newaliascnt{construction}{theorem}

\aliascntresetthe{construction}

\theoremstyle{definition}
\newaliascnt{problem}{theorem}
\newtheorem{problem}[problem]{Problem}
\aliascntresetthe{problem}

\newaliascnt{definition}{theorem}
\newtheorem{definition}[definition]{Definition}
\aliascntresetthe{definition}

\newaliascnt{remark}{theorem}

\aliascntresetthe{remark}

\crefname{theorem}{Theorem}{Theorems}
\crefname{lemma}{Lemma}{Lemmas}
\crefname{problem}{Problem}{Problems}
\crefname{conjecture}{Conjecture}{Conjectures}
\crefname{construction}{Construction}{Constructions}
\crefname{remark}{Remark}{Remarks}
\crefname{corollary}{Corollary}{Corollaries}

\newcommand{\eps}{\varepsilon}
\newcommand{\N}{\mathbb N}

\newcommand{\cP}{\mathcal P}
\newcommand{\cF}{\mathcal F}
\newcommand{\one}{\mathbf 1}
\newcommand{\dd}{\mathop{}\!\mathrm d}

\newcommand{\E}{\mathbb E}
\newcommand{\Prob}{\mathbb P}
\DeclareMathOperator{\Bin}{Bin}

\title[Tight Hamilton Cycles in Uniformly Dense $3$-Graphs]{Sharp Diagonal Thresholds for Tight Hamilton Cycles in Uniformly Dense $3$-Graphs}

\author[H. Lin]{Hao Lin}
\address{School of Mathematical Sciences, University of Science and Technology of China, Hefei, Anhui 230026, China}
\email{haolinz6@ustc.edu.cn}

\author[G. Wang]{Guanghui Wang}
\address{School of Mathematics, Shandong University, Jinan 250100, China}
\email{ghwang@sdu.edu.cn}

\author[W. Zhou]{Wenling Zhou}
\address{School of Mathematics, Shandong University, Jinan 250100, China}
\email{gracezhou@sdu.edu.cn}


\date{}
\begin{document}
%\pagewiselinenumbers

\begin{abstract}
A $3$-uniform hypergraph (or $3$-graph) $H$ on $n$ vertices is
\emph{$(n,d,\mu)$-dense} if
$e_H(X,Y,Z)\ge d|X||Y||Z|-\mu n^3$
for all $X,Y,Z\subseteq V(H)$.  This is one of the weakest standard notions
of quasirandomness for $3$-graphs and is also known as linear quasirandomness.

In this paper, we determine the sharp diagonal thresholds for tight Hamilton cycles in $(n,d,\mu)$-dense $3$-graphs $H$ under conditions on the
minimum vertex degree $\delta_1(H)$ and the minimum codegree $\delta_2(H)$.
We actually prove a general result: define
\[
 f(d):=\frac{1-\sqrt{(4d-1)/3}}2.
\]
We prove that $(n,d,\mu)$-density together with
$\delta_1(H)\ge\alpha\binom{n-1}{2}$ forces a tight Hamilton cycle whenever $d > 1/3$ and
$\alpha>f(d)$. In particular, $f(1/3)=1/3$, which answers
Problem~8.3(i) of Ara\'ujo, Piga and Schacht and confirms Conjecture~8.1 of
Han, Shu and Wang. For the minimum codegree condition, the sharp diagonal threshold is
$(\kappa,\kappa)$, where $\kappa$ is the unique real solution of $\kappa=(1-\kappa)^3$.  Since
$\kappa\approx0.3177>1/4$, this gives a negative answer to Problem~8.3(ii) of
Ara\'ujo, Piga and Schacht and disproves Conjecture~8.2 of Han, Shu and Wang. 
The two proofs use a common Hamilton-framework reduction, but the two degree conditions lead to distinct dominant-component lemmas for $(n, d, \mu)$-dense $3$-graphs, which are of independent interest and whose proofs do not rely on the absorption method.
\end{abstract}

\maketitle

\section{Introduction}\label{sec:introduction}
Hamiltonicity is one of the most fundamental and extensively studied
topics in extremal graph theory.
Since Karp
\cite{Karp1972} proved that deciding whether a graph contains a Hamilton
cycle is NP-complete, it is natural to seek sufficient (and necessary) conditions guaranteeing such cycles. Two prominent classes of
conditions have been intensively investigated: local minimum degree conditions and global quasirandomness conditions.
For graphs, Dirac's theorem
\cite{Dirac1952} states that every graph on $n\ge3$ vertices with minimum
degree at least $n/2$ contains a Hamilton cycle, and the sharpness of the bound is witnessed by the construction consisting of
two disjoint cliques of orders as equal as possible. The blow-up lemma of Koml\'os, S\'ark\"ozy and Szemer\'edi
\cite{KomlosSarkozySzemeredi1997} implies that dense quasirandom graphs with
a mild minimum degree condition contain a broad class of
bounded-degree spanning graphs.

The corresponding problem for hypergraphs is substantially more delicate. Given $k\ge 3$, a \emph{$k$-uniform hypergraph}, or \emph{$k$-graph} for short, is a
hypergraph whose edges all have size $k$. 
Unlike graphs, $k$-graphs admit several inequivalent degree parameters, several natural notions of
cycles, and a genuine hierarchy of quasirandomness properties.
There is now an extensive literature on Dirac-type minimum
$s$-degree conditions for Hamilton cycles in hypergraphs;
see, for example, \cite{KuhnOsthus2006,RodlRucinskiSzemeredi2006,
RodlRucinskiSzemeredi2008,HanSchacht2010,KuhnMycroftOsthus2010,
KeevashKuhnMycroftOsthus2011,RodlRucinskiSzemeredi2011,
BussHanSchacht2013,RodlRucinski2014,HanZhao2015Codegree,
HanZhao2015Vertex,DeOliveiraBastosEtAl2017,
RodlRucinskiSchachtSzemeredi2017,DeOliveiraBastosEtAl2018,
ReiherEtAl2019}. We refer to
\cite{KO-survey-2014,RR-survey-2010,Zhao-survey-2016}
for surveys of this development.
A second line of
research asks which quasirandomness assumptions, supplemented by suitable
local degree conditions, force spanning cycles.  The study of Hamilton cycles under quasirandomness
conditions was initiated by Lenz, Mubayi and Mycroft
\cite{LenzMubayiMycroft2016}.

\subsection{Dirac-type conditions}
Throughout the paper, we focus on $3$-graphs. For
$\ell\in\{1,2\}$, a $3$-graph is an \emph{$\ell$-cycle} if its vertices
admit a cyclic ordering in which each edge consists of three consecutive
vertices and consecutive edges in their natural cyclic order intersect in
exactly $\ell$ vertices.  The cases $\ell=1$ and $\ell=2$ are called a
\emph{loose cycle} and a \emph{tight cycle}, respectively.  A loose or tight
cycle that spans the host $3$-graph is a \emph{loose Hamilton cycle} or a
\emph{tight Hamilton cycle}.  This notion was introduced by Katona and
Kierstead \cite{KatonaKierstead1999}.  We abbreviate the triple
$\{x,y,z\}$ to $xyz$ and the pair $\{u,v\}$ to $uv$.

For a $3$-graph $H$ and distinct vertices $u,v\in V(H)$, define
\[
 \begin{aligned}
 N_H(v)&:=\bigl\{xy\in\tbinom{V(H)\setminus\{v\}}2:vxy\in E(H)\bigr\},
 &\deg_H(v)&:=|N_H(v)|,\\
 N_H(uv)&:=\bigl\{w\in V(H)\setminus\{u,v\}:uvw\in E(H)\bigr\},
 &\deg_H(uv)&:=|N_H(uv)|.
 \end{aligned}
\]
For $s\in\{1,2\}$, the \emph{minimum $s$-degree} of $H$ is
$\delta_s(H):=\min\{\deg_H(S):S\in\binom{V(H)}s\}$.  Thus $\delta_1(H)$ is the
\emph{minimum vertex degree} and $\delta_2(H)$ is the \emph{minimum
codegree}.

The Dirac-type theory of hypergraph Hamilton cycles is now well developed.
For loose Hamilton cycles, Bu\ss, H\`an and Schacht
\cite{BussHanSchacht2013} established the asymptotically optimal minimum
vertex degree threshold $7/16$, and Han and Zhao \cite{HanZhao2015Vertex}
subsequently determined the exact threshold for sufficiently large $n$.  For
tight Hamilton cycles, R\"odl, Ruci\'nski and Szemer\'edi
\cite{RodlRucinskiSzemeredi2006,RodlRucinskiSzemeredi2008,
RodlRucinskiSzemeredi2011} established the sharp asymptotic minimum-codegree threshold $1/2$.
At the vertex degree setting, Reiher, R\"odl, Ruci\'nski, Schacht and
Szemer\'edi \cite{ReiherEtAl2019} proved that the asymptotic threshold for a
tight Hamilton cycle is $(5/9+o(1))\binom{n-1}{2}$. The extremal
constructions behind these results have pronounced global structure.  This
suggests a complementary question: how far can the local degree assumptions
be reduced when the host hypergraph already has a uniform edge distribution?

\subsection{Hamiltonicity in uniformly dense hypergraphs}
The theory of quasirandom graphs was developed in the foundational work
of Chung, Graham and Wilson \cite{ChungGrahamWilson1989}. In graphs, many natural quasirandom properties
are asymptotically equivalent.  For hypergraphs they are not; see
\cite{ChungGraham1990,LenzMubayi2015Poset,AignerHorevEtAl2018}.  We work with
one of the weakest standard notions of quasirandomness.

For a $3$-graph $H$ and not necessarily disjoint sets
$X,Y,Z\subseteq V(H)$, let $e_H(X,Y,Z)$ denote the number of ordered triples
$(x,y,z)\in X\times Y\times Z$ for which $xyz\in E(H)$.  Given $d,\mu>0$,
an $n$-vertex $3$-graph $H$ is \emph{$(d,\mu)$-dense} if
\begin{equation}\label{eq:def-density-intro}
 e_H(X,Y,Z)\ge d|X||Y||Z|-\mu n^3
\end{equation}
for all $X,Y,Z\subseteq V(H)$.  We call $H$ \emph{$(n,d,\mu)$-dense} if it
has $n$ vertices and is $(d,\mu)$-dense.  Hypergraphs satisfying this
one-sided box-discrepancy condition are often called \emph{uniformly dense}
or \emph{linearly quasirandom}.  The corresponding two-sided condition
controls copies of fixed linear hypergraphs, namely hypergraphs in which any
two edges meet in at most one vertex; see
\cite{ConlonEtAl2012,KohayakawaEtAl2010}.

Spanning problems under quasirandomness have been studied from several
directions.  Lenz and Mubayi \cite{LenzMubayi2016} proved perfect-packing
results under linear quasirandomness and a mild minimum vertex degree
condition; related $F$-factor problems were developed further by Ding, Han,
Sun, Wang and Zhou \cite{Ding2022,DingEtAl2023}.  Lenz, Mubayi and Mycroft
\cite{LenzMubayiMycroft2016} showed that positive linear quasirandomness,
together with positive minimum vertex degree, forces a loose Hamilton cycle.
For tight Hamilton cycles, the first positive results required stronger
notions of quasirandomness.  Aigner-Horev and Levy
\cite{AignerHorevLevy2021} treated cherry-quasirandom $3$-graphs under a
minimum codegree condition, and Gan and Han \cite{GanHan2022} obtained the
corresponding minimum vertex degree result.  For the intermediate
vertex--pair density notion, Ara\'ujo, Piga and Schacht
\cite{AraujoPigaSchacht2022} proved that density strictly larger than $1/4$,
together with positive minimum vertex degree, forces a tight Hamilton cycle,
and that $1/4$ is optimal. For non-linear Hamilton cycles in uniformly dense hypergraphs, Han, Shu
and Wang \cite{HanShuWang2021,HanShuWang2026} identified the problem of
connecting prescribed ordered ends as the main obstacle and developed a
general absorbing framework around it.

Against this background, Ara\'ujo, Piga and Schacht \cite[Problem~8.3]{AraujoPigaSchacht2022} asked
whether the following two diagonal conditions force a tight Hamilton cycle
under \eqref{eq:def-density-intro}:
\begin{enumerate}[label=\textup{(\roman*)},leftmargin=2.4em]
 \item the density and the normalized minimum vertex degree are both strictly
       larger than $1/3$;
 \item the density and the normalized minimum codegree are both strictly
       larger than $1/4$.
\end{enumerate}
The same questions appeared as Conjectures~8.1 and~8.2 in the conference
version of Han, Shu and Wang \cite{HanShuWang2021}; see also the discussion
in the subsequent journal version \cite{HanShuWang2026}. Our main results
settle both questions and, in the minimum-vertex-degree setting, determine
the full sharp boundary above density $1/3$.

\subsection{Main results}
For $d\in[1/4,1]$, define
\begin{equation}\label{eq:vertex-boundary}
 f(d):=\frac{1-\sqrt{(4d-1)/3}}2,
 \qquad
 \rho(d):=1-f(d)=\frac{1+\sqrt{(4d-1)/3}}2.
\end{equation}
Then
\begin{equation}\label{eq:rho-identity}
 d=\rho(d)^3+(1-\rho(d))^3=f(d)^3+(1-f(d))^3.
\end{equation}
Our first theorem determines the sharp minimum-vertex-degree boundary
throughout the range $d>1/3$.

\begin{theorem}\label{thm:main-vexdeg-curve}
For every $d\in(1/3,1]$ and every $\alpha>f(d)$, there exist $\mu>0$ and
$n_0\in\N$ such that every $(n,d,\mu)$-dense $3$-graph $H$ with
$n\ge n_0$ and $\delta_1(H)\ge\alpha\binom{n-1}{2}$
contains a tight Hamilton cycle.
\end{theorem}

The condition in \cref{thm:main-vexdeg-curve} is asymptotically sharp for
every $d>1/3$.  Indeed, the construction in \cref{subsec:const-1},
originally introduced in
\cite{AraujoPigaSchacht2022,HanShuWang2021}, gives an
$(n,d,\mu)$-dense $3$-graph with normalized minimum vertex degree
$f(d)-o(1)$ and no tight Hamilton cycle.  Since $f(d)<1/3$ for $d>1/3$
and $f(d)\to1/3$ as $d\downarrow1/3$, the curve has diagonal endpoint
$(1/3,1/3)$.

For reference, we retain the diagonal formulation as the following
immediate consequence, since \cref{cor:main-vexdeg} solves
\cite[Problem~8.3(i)]{AraujoPigaSchacht2022} and confirms
\cite[Conjecture~8.1]{HanShuWang2021}.

\begin{corollary}\label{cor:main-vexdeg}
\label{coro:main-vexdeg}
For every $d,\alpha> 1/3$, there exist $\mu>0$ and
$n_0\in\N$ such that every $(n,d,\mu)$-dense $3$-graph $H$ with
$n\ge n_0$ and
$\delta_1(H)\ge\alpha\binom{n-1}{2}$ contains a tight Hamilton cycle.
\end{corollary}

However, the minimum codegree problem has a different obstruction.  Let
$\kappa\in(1/4,1/3)$ be the unique solution of
$\kappa=(1-\kappa)^3$; numerically, $\kappa \approx 0.3177$.

\begin{theorem}\label{thm:codegree-counterexample}
For every $\eps,\mu>0$, there exists $n_0\in\N$ such that, for every integer
$n\ge n_0$, there is an $(n,\kappa,\mu)$-dense $3$-graph $H$ with
$\delta_2(H)\ge(\kappa-\eps)n$ and no tight Hamilton cycle.
\end{theorem}

Since $\kappa>1/4$, \cref{thm:codegree-counterexample} gives a negative
answer to \cite[Problem~8.3(ii)]{AraujoPigaSchacht2022} and disproves
\cite[Conjecture~8.2]{HanShuWang2021}. 

The next result shows that this obstruction is sharp on the diagonal.

\begin{theorem}\label{thm:main-codeg}
For every $d,\alpha>\kappa$, there exist $\mu>0$ and $n_0\in\mathbb N$ such that every $(n,d,\mu)$-dense $3$-graph $H$ with $n\ge n_0$ and $\delta_2(H)\ge\alpha n$
contains a tight Hamilton cycle.
\end{theorem}

\subsection{Ideas of the proofs}

The lower bounds follow from two distinct probabilistic constructions;
see \cref{prop:vertex-obstruction,prop:codegree-obstruction}. In both cases, randomness guarantees uniform density and the required
minimum degree, whereas deterministic obstructions rule out a tight
Hamilton cycle.  Both constructions separate the relevant edge classes
into distinct tight components, and the codegree construction additionally
uses a space barrier.  The codegree construction, which gives the critical
value $\kappa$, is new. 

The main novelty lies in the proofs of
\cref{thm:main-vexdeg-curve,thm:main-codeg}.  Here conditions are too
weak to directly provide the prescribed-end tight connectors required by the usual
absorption method. Instead, we employ a recently developed Hamilton framework due to Lang
and Sanhueza-Matamala~\cite{LangSanhueza2026} and reduce both theorems
to finding, inside almost every induced subgraph of a sufficiently large
fixed order, a local Hamilton framework supported by a spanning tight
component. To locate such a component, we
color each pair in the shadow by the tight component containing it and prove
two essentially different dominant-component lemmas. The dominant components then give the local
Hamilton frameworks in
\cref{lem:local-vdeg-framework,lem:local-codeg-framework}, which together with the inheritance lemmas and the
Hamilton-framework embedding theorem, complete the proofs of
\cref{thm:main-vexdeg-curve,thm:main-codeg}.

In the minimum vertex degree setting, the main obstacle is that the
component coloring may use an unbounded number of colors.  We first apply
the Kruskal--Katona theorem to ignore components with small shadows and then use the multicolor regularity lemma to obtain a finite weighted coloring of a
reduced graph.  The key new ingredient is the weighted monochromatic-clique
theorem (\cref{thm:weighted-clique}): if the total monochromatic weight of every triangle exceeds some
$t\ge1/3$, then one color has weight greater than $t$ on every pair.
Its robust form forces one tight component to dominate almost all reduced
pairs.  After transferring this information back to the original
$3$-graph, we obtain a dominant-component lemma (\cref{lem:v-dominant}): for
$\rho=\rho(d)$, one spanning tight component has shadow density about
$\rho$, uniform density about $\rho^3$, and normalized minimum vertex
degree at least $\alpha+\rho-1-o(1)$. This quantity is positive precisely when
$\alpha>1-\rho(d)=f(d)$, thereby explaining the sharp boundary in
\cref{thm:main-vexdeg-curve}.

The minimum codegree setting requires a different argument, since
$\kappa<1/3$ lies below the range of the weighted clique theorem. Here the
codegree condition gives strong common-neighbor information inside every
component shadow.  After transferring this information to the reduced
weighted graph, a stability analysis around
$\kappa=(1-\kappa)^3$
together with a robust two-color common-neighbor inequality forces one
component to dominate and be spanning; see \cref{lem:codeg-dominant}.
The consistency condition in the Hamilton framework is supplied
separately by the sharp shadow-consistency lemma,
\cref{lem:shadow-consistency}.

\subsection{Organization}
This paper is organized as follows.
\Cref{sec:preliminaries} collects the preliminary notions, auxiliary
results, and tools used throughout the paper. The two sharpness constructions are
given in~\cref{sec:constructions}. In~\cref{sec:framework}, we
state the main structural lemmas, construct local Hamilton frameworks
in the minimum-vertex-degree and minimum-codegree settings, and deduce
\cref{thm:main-vexdeg-curve,thm:main-codeg} from these structural results;
the diagonal statement \cref{coro:main-vexdeg} follows immediately.
The proof of the main lemma in the minimum-vertex-degree case is given
in~\cref{sec:pf-of-v-dominant}, together with the weighted
monochromatic-clique theorem required for its proof. The two main lemmas
in the minimum-codegree case are proved in~\cref{sec:pf-of-co-dominant}.
We conclude with further remarks in~\cref{sec:remark}.

\section{Preliminaries}\label{sec:preliminaries}

For a positive integer $n$, write $[n]:=\{1,\ldots,n\}$. For a graph $G$
and sets $X,Y\subseteq V(G)$, let
\[
 e_G(X,Y):=|\{(x,y)\in X\times Y:x\ne y,\ xy\in E(G)\}|.
\]
In particular,
$e_G(X,X)=2e(G[X])$. We write $e_G(X):=e(G[X])$ and $\overline G$ for
the complement of $G$ on $V(G)$.

Two edges of a $3$-graph are \emph{tightly adjacent} if they share a pair.
A \emph{tight component} of a $3$-graph $H$ is an edge-maximal tightly
connected subgraph, and it is \emph{spanning} if its vertex set is $V(H)$.
The \emph{shadow} of $H$ is
\[
 \partial H:=\bigl\{xy\in\tbinom{V(H)}2:
 xyz\in E(H)\text{ for some }z\in V(H)\bigr\}.
\]
The shadows of distinct tight components are edge-disjoint.

\subsection{Hamilton frameworks}
We begin with the Hamilton-framework formalism of Lang and
Sanhueza-Matamala~\cite{LangSanhueza2026}, in the specialized form used
throughout the paper.
A \emph{perfect fractional matching} in a $3$-graph $H$ is a function
$\omega:E(H)\to[0,1]$ such that for each $v\in V(H)$ we have $\sum_{v\in e\in E(H)}\omega(e)=1$. A \emph{closed walk of order $m$} in $H$ is the image of a
homomorphism $C_m^{(3)}\to H$, where $C_m^{(3)}$ denotes the tight cycle of
order $m$.

\begin{definition}
Let $\cP_s$ be a family of $s$-vertex $3$-graphs. We say
that $\cP_s$ \emph{admits a Hamilton framework} if, for every
$H\in\cP_s$, one can choose a subgraph $\cF(H)\subseteq H$ satisfying
\begin{enumerate}[label=(F\arabic*),leftmargin=3.2em]
 \item\label{item:F1} $\cF(H)$ is a spanning tight component of $H$;
 \item\label{item:F2} $\cF(H)$ has a perfect fractional matching;
 \item\label{item:F3} $\cF(H)$ has a closed walk of order $1$ modulo $3$;
 \item\label{item:F4} whenever $J$ is an $(s+1)$-vertex $3$-graph such that
 $H:=J-x$ and $H':=J-y$ both belong to $\cP_s$ for distinct $x,y\in V(J)$, the union
 $\cF(H)\cup\cF(H')$ is tightly connected in $J$.
\end{enumerate}
\end{definition}

These conditions correspond to connectivity, space, aperiodicity, and consistency, respectively. The following lemma shows that their robust local presence suffices to guarantee Hamiltonicity, which is the case $k=t=3$ of the general hypergraph bandwidth theorem of
Lang and Sanhueza-Matamala
\cite[Theorem~3.5]{LangSanhueza2026}.

\begin{lemma}[Embedding lemma for Hamilton frameworks]
\label{lem:HBT-special}
Fix $s\ge6$, and let $\cP_s$ be a family of $s$-vertex $3$-graphs that
admits a Hamilton framework. There exists $n_0\in\N$ such that the
following holds. Let $H$ be an $n$-vertex $3$-graph with $n\ge n_0$.
Suppose that, for every $6$-set $R\subseteq V(H)$, at least $\left(1-\frac1{s^2}\right)\binom{n-6}{s-6}$
sets $S\in\binom{V(H)}s$ with $R\subseteq S$ satisfy $H[S]\in\cP_s$.
Then $H$ contains a tight Hamilton cycle.
\end{lemma}

\subsection{Inheritance and extremal tools}

We next record the two inheritance statements needed to apply
\cref{lem:HBT-special}. The density statement is from
\cite[Lemma~8.1]{Lang2026}, and the degree statement is from~\cite[Lemma~4.3]{LangSanhueza2026}.

\begin{lemma}\label{lem:density-inheritance}
For $1/r,d,\mu'\gg1/s\gg\mu\gg1/n$, let $H$ be an
$(n,d,\mu)$-dense $3$-graph. Then, for every
$R\in\binom{V(H)}r$, all but at most
$e^{-s^{1/6}}\binom{n-r}{s-r}$ sets
$S\in\binom{V(H)}s$ with $R\subseteq S$ induce an
$(s,d,\mu')$-dense $3$-graph.
\end{lemma}

\begin{lemma}\label{lem:degree-inheritance}
For $1/r,\eps\gg1/s\gg1/n$, let $\ell\in\{1,2\}$ and $\alpha\ge0$.
If an $n$-vertex $3$-graph $H$ satisfies $\delta_\ell(H)\ge(\alpha+\eps)
 \binom{n-\ell}{3-\ell}$,
then for every $R\in\binom{V(H)}r$, all but at most
$e^{-\sqrt s}\binom{n-r}{s-r}$ sets
$S\in\binom{V(H)}s$ with $R\subseteq S$ satisfy
$\delta_\ell(H[S])\ge\alpha\binom{s-\ell}{3-\ell}$.
\end{lemma}

The next two results provide the space and aperiodicity conditions in the
Hamilton framework. The first is the following perfect-matching consequence of a
perfect-packing theorem of Lenz and Mubayi \cite[Theorem~3]{LenzMubayi2016}. The second follows from the
determination of the uniform Tur\'an density of tight cycles by Buci\'c,
Cooper, Kr\'a\v{l}, Mohr, and Munh\'a Correia
\cite[Theorem~1.1]{BucicEtAl2023}.

\begin{lemma}\label{lem:perfect-matching}
For every $d,\alpha>0$, there exist $\mu>0$ and $n_0\in\N$ such that
every $(n,d,\mu)$-dense $3$-graph $H$ with $n\ge n_0$, $3\mid n$, and
$\delta_1(H)\ge\alpha\binom{n-1}{2}$ has a perfect matching.
\end{lemma}

\begin{lemma}\label{lem:C7}
For every $d>4/27$, there exist $\mu>0$ and $n_0\in\N$ such that every
$(n,d,\mu)$-dense $3$-graph with $n\ge n_0$ contains a copy of
$C_7^{(3)}$.
\end{lemma}

Finally, we use the Lov\'asz form of the Kruskal--Katona theorem
\cite[Theorem~1]{Keevash2008}.

\begin{lemma}\label{lem:kk-triangle}
Every graph $G$ with $m$ edges contains at most
$(2m)^{3/2}/6$ triangles.
\end{lemma}

\subsection{Regularity tools}
In this section, we collect the regularity tools required for the proofs of the key lemmas.
Let $G$ be a graph and let $A,B\subseteq V(G)$ be disjoint and nonempty.
Write $d_G(A,B):=e_G(A,B)/(|A||B|)$. The pair $(A,B)$ is
$\eta$-regular in $G$ if
$|d_G(A',B')-d_G(A,B)|\le\eta$ whenever
$A'\subseteq A$, $B'\subseteq B$, $|A'|\ge\eta|A|$, and
$|B'|\ge\eta|B|$. We now state a version of Szemer\'edi regularity lemma for 
edge-colored complete graphs. 

\begin{lemma}[Multicolor regularity lemma]\label{lem:multicolor-regularity}
For every $\eta>0$ and integers $L,m_0\ge1$, there exist
$M,n_0\in\N$ such that the following holds. For every $L$-edge-colored complete graph $G$ on 
$n\ge n_0$ vertices, there is a partition
\[
 V=V_0\cup V_1\cup\cdots\cup V_m
\]
such that $m_0\le m\le M$, $|V_0|\le\eta n$, and
$|V_1|=\cdots=|V_m|$, and all but at most $\eta m^2$ pairs
$\{a,b\}\in\binom{[m]}2$ are $\eta$-regular simultaneously in every monochromatic  
$G_i$ for $i\in[L]$.
\end{lemma}

The usual energy-increment proof applies to the sum of the $L$ color
energies, so the multicolor statement follows from the ordinary regularity
lemma.

We repeatedly use the following two elementary consequences of regularity.

\begin{lemma}[Weighted cut estimate]\label{lem:regular-cut}
Let $(A,B)$ be an $\eta$-regular pair in a graph $G$, and put
$p:=d_G(A,B)$. Then, for all functions $f:A\to[0,1]$ and
$g:B\to[0,1]$,
\[
 \left|\sum_{x\in A}\sum_{y\in B}
 \bigl(\one_G(xy)-p\bigr)f(x)g(y)\right|
 \le2\eta|A||B|.
\]
\end{lemma}

\begin{proof}
First take $f=\one_X$ and $g=\one_Y$. If
$|X|\ge\eta|A|$ and $|Y|\ge\eta|B|$, regularity gives the result with
error at most $\eta|A||B|$; if either set is smaller, the trivial bound is
at most $2\eta|A||B|$. The general case follows by integrating the
indicator functions of the level sets of $f$ and $g$.
\end{proof}

\begin{lemma}[Weighted triangle counting]\label{lem:regular-triangle}
Let $A_1,A_2,A_3$ be pairwise disjoint vertex sets. For
$ab\in\{12,13,23\}$, let $G_{ab}$ be a graph between $A_a$ and $A_b$,
and suppose that $(A_a,A_b)$ is $\eta$-regular in $G_{ab}$ with density
$p_{ab}$. Then, for all $f_i:A_i\to[0,1]$,
\[
\begin{split}
 \biggl|&\sum_{x_i\in A_i}\prod_{i=1}^3f_i(x_i)
        \prod_{ab\in\{12,13,23\}}\one_{G_{ab}}(x_a x_b)-p_{12}p_{13}p_{23}
          \prod_{i=1}^3\sum_{x\in A_i}f_i(x)\biggr|
 \le6\eta|A_1||A_2||A_3|.
\end{split}
\]
\end{lemma}

\begin{proof}
Replace the three edge indicators one at a time by their densities. After
fixing the remaining vertex, each resulting error is bounded by
\cref{lem:regular-cut} and is therefore at most
$2\eta|A_1||A_2||A_3|$.
\end{proof}

\subsection{Reductions for component colorings}

The next estimate allows us to discard component colors whose shadows have
vanishing density.

\begin{lemma}[Tail of the component coloring]\label{lem:tail-triangles}
Let $G_1,G_2,\ldots$ be pairwise edge-disjoint graphs on an $n$-vertex
set, listed so that $e(G_1)\ge e(G_2)\ge\cdots$. Then, for every $L\ge1$,
\[
 \sum_{i>L}\#K_3(G_i)=O\left(\frac{n^3}{\sqrt L}\right),
\]
where the implied constant is absolute.
\end{lemma}

\begin{proof}
Put $m_i:=e(G_i)$. By \cref{lem:kk-triangle},
\[
 \sum_{i>L}\#K_3(G_i)
 \le\frac{2^{3/2}}6\sum_{i>L}m_i^{3/2}
 \le\frac{2^{3/2}}6\sqrt{m_{L+1}}\sum_{i>L}m_i.
\]
Since $\sum_i m_i\le\binom n2$ and
$m_{L+1}\le\binom n2/(L+1)$, the result follows.
\end{proof}

For the codegree argument, every edge in a component shadow has many common
neighbors in the same shadow. The following lemma transfers this property
to the reduced graph. The diagonal cluster indices are included explicitly.

\begin{lemma}[Finite support transfer]\label{lem:finite-support-transfer}
Let $0<\eta<1/2$, $q\in[0,1]$, $\gamma>0$, and $\theta\ge0$.
Let $G$ be a graph on $n$ vertices such that every edge of $G$ has at
least $qn$ common neighbors in $G$. Let
$V(G)=V_0\cup V_1\cup\cdots\cup V_m$, where
$|V_0|\le\eta n$ and $|V_1|=\cdots=|V_m|$. For $a\ne b$, put
$w(ab):=d_G(V_a,V_b)$. Suppose that $w(ab)\ge\gamma$, the pair
$(V_a,V_b)$ is $\eta$-regular in $G$, and at most $\theta m$ indices
$c\in[m]\setminus\{a,b\}$ make either $(V_a,V_c)$ or $(V_b,V_c)$
nonregular. Then
\begin{equation}\label{eq:finite-support-transfer}
 \frac1m\sum_{c=1}^m w(ac)w(bc)
 \ge q-2\eta-\frac{6\eta+\theta+2/m}{\gamma}.
\end{equation}
The diagonal values $w(aa)$ and $w(bb)$ may be chosen arbitrarily in
$[0,1]$.
\end{lemma}

\begin{proof}
Let $s:=|V_1|$ and $p:=w(ab)$. Count triples $(x,y,z)$ such that
$x\in V_a$, $y\in V_b$, $xy\in E(G)$, and $z$ is a common neighbor of
$x$ and $y$ in $G$. The hypothesis gives at least $qnp s^2$ such triples.
Those with $z\in V_0$ or $z\in V_a\cup V_b$ contribute at most
$|V_0|ps^2+2s^3$, and the nonregular indices contribute at most
$\theta ms^3$. For every remaining index, \cref{lem:regular-triangle}
gives the upper bound
$\bigl(pw(ac)w(bc)+6\eta\bigr)s^3$. Hence
\[
 qnp s^2\le |V_0|ps^2+2s^3
 +ps^3\sum_{c=1}^m w(ac)w(bc)+(6\eta+\theta)ms^3.
\]
Divide by $pms^3$ and put $\lambda:=|V_0|/(ms)$. Since
$\lambda\le\eta/(1-\eta)\le2\eta$ and $p\ge\gamma$, we obtain
\eqref{eq:finite-support-transfer}.
\end{proof}
\section{Sharpness constructions}\label{sec:constructions}

We prove the sharpness of \cref{thm:main-vexdeg-curve} and then establish
\cref{thm:codegree-counterexample}. Both constructions are probabilistic,
but their obstructions to Hamiltonicity are deterministic.

We use the following concentration inequalities. Part~\ref{item-ineq1} is McDiarmid's bounded-differences inequality
\cite{McDiarmid1989}, and part~\ref{item-ineq2} is the additive Chernoff
bound.

\begin{lemma}\label{lem:concentration}
Let $\xi_1,\ldots,\xi_m$ be independent random variables.
\begin{enumerate}[label=\textup{(\alph*)},leftmargin=2.4em]
 \item\label{item-ineq1} Suppose that $Z=f(\xi_1,\ldots,\xi_m)$ and that
 changing only $\xi_i$ changes $Z$ by at most $c_i$. Then, for every
 $\lambda>0$,
 \[
  \Prob(Z\le\E Z-\lambda)
  \le\exp\left(-\frac{2\lambda^2}{\sum_{i=1}^m c_i^2}\right).
 \]
 \item\label{item-ineq2} If $X\sim\Bin(N,p)$, then, for every
 $0\le\lambda\le Np$,
 \[
  \Prob(X\le Np-\lambda)
  \le\exp\left(-\frac{\lambda^2}{2Np}\right).
 \]
\end{enumerate}
\end{lemma}

\subsection{The vertex-degree obstruction}\label{subsec:const-1}

For $\theta\in[1/2,1]$, define
\begin{equation}\label{eq:theta-parameters}
 d_\theta:=\theta^3+(1-\theta)^3,
 \qquad
 \alpha_\theta:=\min\{1-\theta,d_\theta\}.
\end{equation}

\begin{lemma}[Vertex-degree obstruction]\label{prop:vertex-obstruction}
For every $\theta\in[1/2,1]$ and every $\eps,\mu>0$, there exists
$n_0\in\N$ such that, for every $n\ge n_0$, there is an
$(n,d_\theta,\mu)$-dense $3$-graph $H$ satisfying
$\delta_1(H)\ge(\alpha_\theta-\eps)\binom{n-1}{2}$
and containing no tight Hamilton cycle. In particular, for every
$d\in(1/3,1]$, taking $\theta=\rho(d)$ gives $d_\theta=d$ and
$\alpha_\theta=f(d)$, proving the asymptotic sharpness of
\cref{thm:main-vexdeg-curve}. Taking $\theta=2/3$ gives the diagonal
point $(1/3,1/3)$ in \cref{coro:main-vexdeg}.
\end{lemma}

\begin{proof}
Let $U$ be a set of size $N:=n-2$. Color each pair in $\binom U2$
independently red with probability $\theta$ and blue otherwise. Add two
vertices $x$ and $y$, and define a $3$-graph $H$ on
$U\cup\{x,y\}$ as follows:
\begin{itemize}[leftmargin=2em]
 \item a triple contained in $U$ is an edge if its three pairs have the
 same color;
 \item $xuv$ is an edge if $uv$ is red, and $yuv$ is an edge if $uv$ is
 blue;
 \item no edge contains the pair $xy$.
\end{itemize}

Partition $E(H)$ into the red class, consisting of the all-red triples in
$U$ and the edges through $x$, and the blue class, defined analogously.
No red edge is tightly adjacent to a blue edge: a shared pair in $U$ would
have to receive both colors, while red edges avoid $y$ and blue edges avoid
$x$. Thus every tight cycle lies in one class. A red tight cycle avoids
$y$, and a blue tight cycle avoids $x$, so $H$ has no tight Hamilton cycle.

We next verify uniform density. Fix $X,Y,Z\subseteq V(H)$ and set
$D:=e_H(X,Y,Z)$. All but at most $12n^2$ ordered triples in
$X\times Y\times Z$ have three distinct coordinates in $U$, and each such
triple is an edge with probability $d_\theta$. Hence
\begin{equation}\label{eq:v-obstruction-expectation}
 \E D\ge d_\theta|X||Y||Z|-12n^2.
\end{equation}
Changing one pair color affects at most $n$ hyperedges, so it changes $D$
by at most $6n$. Since there are fewer than $n^2/2$ independent pair
colors, part~\ref{item-ineq1} of \cref{lem:concentration} gives
\[
 \Prob\left(D<\E D-\frac{\mu n^3}{2}\right)
 \le\exp\left(-\frac{\mu^2n^2}{36}\right).
\]
There are $8^n$ ordered choices of $(X,Y,Z)$. The union bound and
\eqref{eq:v-obstruction-expectation} therefore show that, with probability
$1-o(1)$,
\[
 e_H(X,Y,Z)\ge d_\theta|X||Y||Z|-\mu n^3
\]
simultaneously for all $X,Y,Z\subseteq V(H)$, provided that $n$ is
sufficiently large.

It remains to estimate the minimum vertex degree. Chernoff's inequality
gives, with probability $1-o(1)$,
\[
 \deg_H(x)\ge\left(\theta-\frac\eps{10}\right)\binom N2,
 \qquad
 \deg_H(y)\ge\left(1-\theta-\frac\eps{10}\right)\binom N2.
\]
For $v\in U$, let $T_v$ be the number of monochromatic triples in
$\binom U3$ containing $v$. Then
$\E T_v=d_\theta\binom{N-1}{2}$. Changing a pair incident with $v$
changes $T_v$ by at most $N-2$, and changing any other pair changes it by
at most one. Another application of part~\ref{item-ineq1} of
\cref{lem:concentration}, followed by a union bound over
$v\in U$, gives, with probability $1-o(1)$,
\[
 T_v\ge d_\theta\binom{N-1}{2}-\frac\eps{10}n^2
 \qquad\text{for every }v\in U.
\]
Moreover, for every $u\in U\setminus\{v\}$, exactly one of $xuv$ and
$yuv$ is an edge, so $\deg_H(v)=T_v+N-1$. Since $\theta\ge1-\theta$,
these estimates imply, for all sufficiently large $n$,
\[
 \delta_1(H)\ge
 \bigl(\min\{1-\theta,d_\theta\}-\eps\bigr)\binom{n-1}{2}
 =(\alpha_\theta-\eps)\binom{n-1}{2}
\]
with probability $1-o(1)$. The density and degree events therefore hold
simultaneously for at least one coloring.
\end{proof}

\subsection{The codegree obstruction}\label{subsec:const-2}

For $\theta\in(0,1/3)$, define
\begin{equation}\label{eq:codegree-obstruction-parameter}
 \beta_\theta:=\min\{\theta,(1-\theta)^3\}.
\end{equation}

\begin{lemma}[Codegree obstruction]\label{prop:codegree-obstruction}
For every $\theta\in(0,1/3)$ and every $\eps,\mu>0$, there exists
$n_0\in\N$ such that, for every $n\ge n_0$, there is an
$(n,\beta_\theta,\mu)$-dense $3$-graph $H$ satisfying
$\delta_2(H)\ge(\beta_\theta-\eps)n$ and containing no tight Hamilton
cycle. In particular, taking $\theta=\kappa$ proves
\cref{thm:codegree-counterexample}.
\end{lemma}

\begin{proof}
Partition an $n$-vertex set as $V=A\cup B$, where
$|B|=\lfloor\theta n\rfloor$ and $|A|=n-|B|$. Choose a random graph $G$
on $A$ by including each pair independently with probability $\theta$.
Define a random $3$-graph $H$ on $V$ as follows:
\begin{itemize}[leftmargin=2em]
 \item a triple contained in $A$ is an edge if none of its three pairs
 belongs to $E(G)$;
 \item $xyw$ is an edge for $x,y\in A$ and $w\in B$ if $xy\in E(G)$;
 \item every triple containing at least two vertices of $B$ is an edge.
\end{itemize}

Let $\mathcal E_A:=E(H)\cap\binom A3$ and
$\mathcal E_B:=E(H)\setminus\mathcal E_A$. No edge in $\mathcal E_A$ is
tightly adjacent to an edge in $\mathcal E_B$: a shared pair would lie in
$A$, where it would have to be both a nonedge and an edge of $G$. Hence
every tight cycle lies in one class. A cycle in $\mathcal E_A$ avoids
$B$. In a cycle in $\mathcal E_B$, every three consecutive vertices
contain a vertex of $B$, so each gap between consecutive vertices of $B$
contains at most two vertices of $A$. A spanning cycle would therefore
imply $|A|\le2|B|$, contrary to $\theta<1/3$. Thus $H$ contains no tight
Hamilton cycle.

Fix $X,Y,Z\subseteq V(H)$ and let $D:=e_H(X,Y,Z)$. At most $3n^2$
ordered triples in $X\times Y\times Z$ have a repeated coordinate. A
triple with distinct coordinates is an edge with probability
$(1-\theta)^3$ if it lies in $A$, with probability $\theta$ if it has
exactly two vertices in $A$, and with probability one if it has at least
two vertices in $B$. Consequently,
\begin{equation}\label{eq:codegree-obstruction-expectation}
 \E D\ge\beta_\theta|X||Y||Z|-3n^2.
\end{equation}
Changing one edge indicator of $G$ changes $D$ by at most $6n$.
McDiarmid's inequality and a union bound over the $8^n$ choices of
$(X,Y,Z)$ therefore show that, with probability $1-o(1)$,
\[
 e_H(X,Y,Z)\ge\beta_\theta|X||Y||Z|-\mu n^3
\]
simultaneously for all $X,Y,Z\subseteq V(H)$.

We now estimate the codegrees. Pairs contained in $B$ have codegree
$n-2$. If $x\in A$ and $y\in B$, then
$\deg_H(xy)=|B|-1+\deg_G(x)\ge|B|-1$. If $x,y\in A$ and
$xy\in E(G)$, then $\deg_H(xy)=|B|$. Thus these pairs have codegree at
least $(\theta-\eps)n$ for sufficiently large $n$.

Suppose that $x,y\in A$ and $xy\notin E(G)$. Conditional on this event,
$z\in A\setminus\{x,y\}$ completes $xy$ to an edge precisely when both
$xz$ and $yz$ are nonedges of $G$. Hence
\[
 \deg_H(xy)\sim\Bin\bigl(|A|-2,(1-\theta)^2\bigr).
\]
For sufficiently large $n$, its expectation is at least
$((1-\theta)^3-\eps/2)n$. By Chernoff's inequality and a union bound over
$\binom A2$, with probability $1-o(1)$,
\[
 \deg_H(xy)\ge\bigl((1-\theta)^3-\eps\bigr)n
\]
simultaneously for every such pair. Therefore, with probability $1-o(1)$,
\[
 \delta_2(H)\ge(\beta_\theta-\eps)n.
\]
At least one realization satisfies this bound and the uniform-density
condition.

It remains to optimize the parameter. The functions $\theta$ and
$(1-\theta)^3$ are respectively increasing and decreasing, so
$\beta_\theta$ is maximized at their unique point of intersection,
namely $\theta=\kappa$. More generally, if
$|B|=(b+o(1))n$ and the auxiliary graph on $A$ has edge probability
$q$, then the density bottlenecks are $q$ and $(1-q)^3$, while the
codegree bottlenecks are $b$ and $(1-q)^2(1-b)$. Subject to $b<1/3$,
\[
 \max_{\substack{0<b<1/3\\0<q<1}}
 \min\bigl\{q,(1-q)^3,b,(1-q)^2(1-b)\bigr\}=\kappa.
\]
Indeed, the minimum is at most $\min\{q,(1-q)^3\}\le\kappa$, and the
choice $b=q=\kappa$ makes all four terms equal to $\kappa$.
\end{proof}
\section{The Hamilton-framework reduction}\label{sec:framework}

In this section, we reduce
\cref{thm:main-vexdeg-curve,thm:main-codeg} to
\cref{lem:v-dominant,lem:codeg-dominant,lem:shadow-consistency}.

\subsection{Structural lemmas}

Recall from \eqref{eq:vertex-boundary}--\eqref{eq:rho-identity} that, for
$d>1/3$, we have $\rho(d)>2/3$, $f(d)=1-\rho(d)$, and
$d=\rho(d)^3+(1-\rho(d))^3$.

\begin{lemma}[Vertex-degree dominant component]\label{lem:v-dominant}
Let $d\in(1/3,1]$ and $\alpha>f(d)$, and put $\rho:=\rho(d)$. For every
$\eps>0$, there exist $\mu>0$ and $n_0\in\N$ such that every
$(n,d,\mu)$-dense $3$-graph $H$ with $n\ge n_0$ and
$\delta_1(H)\ge\alpha\binom{n-1}{2}$ contains a spanning tight component
$C$ satisfying
\begin{enumerate}[label=\textup{(\roman*)},leftmargin=2.4em]
 \item\label{item:p1} $e(\partial C)\ge(\rho-\eps)\binom n2$;
 \item\label{item:p2} for all $X,Y,Z\subseteq V(H)$,
 $e_C(X,Y,Z)\ge(\rho^3-\eps)|X||Y||Z|-\eps n^3$;
 \item\label{item:p3}
 $\delta_1(C)\ge(\alpha+\rho-1-\eps)\binom{n-1}{2}$.
\end{enumerate}
\end{lemma}

The global shadow estimate in part~\ref{item:p1} is used for consistency.
Parts~\ref{item:p2} and~\ref{item:p3} provide the uniform density and
positive minimum vertex degree needed for the space condition, while
part~\ref{item:p2} also supplies the aperiodic closed walk.
The proof is given in \cref{sec:pf-of-v-dominant}.

\begin{lemma}[Codegree dominant component]\label{lem:codeg-dominant}
Let $d,\alpha>\kappa$. There exist $\tau>0$
with the following property. For every $\eps>0$, there exist $\mu>0$ and
$n_0\in\N$ such that every $(n,d,\mu)$-dense $3$-graph $H$ with
$n\ge n_0$ and $\delta_2(H)\ge\alpha n$ contains a spanning tight component $C$
satisfying
\begin{enumerate}[label=\textup{(\roman*)},leftmargin=2.4em]
 \item for all $X,Y\subseteq V(H)$, $e_{\partial C}(X,Y)\ge (\kappa + \tau)|X||Y|-\eps n^2$;
 \item for all $X,Y,Z\subseteq V(H)$,
$e_C(X,Y,Z)\ge\tau|X||Y||Z|-\eps n^3$;
 \item $\delta_1(C)\ge\tau\binom{n-1}{2}$;
 \item $C$ contains a copy of $C_7^{(3)}$.
\end{enumerate}
\end{lemma}

The consistency step in the codegree setting rests on the following sharp
shadow lemma.

\begin{lemma}[Shadow consistency]\label{lem:shadow-consistency}
For every $t>0$ and $q>1/4$, there exist $\eps>0$ and $s_0\in\N$ such
that the following holds. Let $R$ and $B$ be edge-disjoint graphs on a
common vertex set of size $s\ge s_0$. Suppose that, for all
$X,Y\subseteq V(R)=V(B)$,
\[
 e_R(X,Y),e_B(X,Y)\ge t|X||Y|-\eps s^2,
\]
and every edge of $R$ has at least $qs$ common neighbors in $R$, while
every edge of $B$ has at least $qs$ common neighbors in $B$. Then no such
pair $(R,B)$ exists.
\end{lemma}

The proofs of \cref{lem:codeg-dominant,lem:shadow-consistency} are given in
\cref{sec:pf-of-co-dominant}.

\subsection{Local Hamilton frameworks}\label{subsct:loc-1}

We now convert the dominant components into local Hamilton frameworks.

\begin{lemma}[Vertex-degree local Hamilton framework]
\label{lem:local-vdeg-framework}
For every $d\in(1/3,1]$ and $\alpha>f(d)$, there exist $\mu>0$ and
$s_0\in\N$ such that, for every $s\ge s_0$, the family $\cP_s$ of all
$(s,d,\mu)$-dense $3$-graphs $H$ satisfying
$\delta_1(H)\ge\alpha\binom{s-1}{2}$ admits a Hamilton framework.
\end{lemma}

\begin{proof}
Put $\rho:=\rho(d)$,
$d_*:=\rho^3/2>4/27$, and
$\alpha_*:=(\alpha+\rho-1)/2>0$. Apply
\cref{lem:perfect-matching} with $d_*,\alpha_*$, obtaining
$\mu_{\rm pm}>0$, and apply \cref{lem:C7} with
$d_*$, obtaining $\mu_7>0$. Choose
$\eps>0$ so small that
\[
 \rho^3-\eps>d_*,\qquad
 \alpha+\rho-1-\eps>\alpha_*,\qquad
 \rho-\eps>\frac12,
\]
and $\eps<\min\{\mu_{\rm pm},\mu_7\}$. Apply
\cref{lem:v-dominant} with error $\eps/2$, and increase the resulting
$s_0$ so that all order requirements below are satisfied.

For each $H\in\cP_s$, choose a component $C(H)$ supplied by
\cref{lem:v-dominant} and set $\cF(H):=C(H)$. Fix $H\in\cP_s$ and write
$C:=C(H)$. We verify the
four framework conditions.

Condition~\ref{item:F1} holds since $\cF(H)=C$ is a spanning tight
component of $H$.

{\bf Verifying~\ref{item:F2}.}
Let $C(3)$ be the balanced blow-up obtained by replacing every vertex
$v\in V(C)$ with a cluster $V_v$ of size $3$, and every edge of $C$ with
the complete tripartite $3$-graph on the corresponding clusters. For $Y_i\subseteq V(C(3))$, where $i\in[3]$, define
$a_i(v):=|Y_i\cap V_v|/3$. Then
$\sum_{v\in V(C)}a_i(v)=|Y_i|/3$, and
\[
 e_{C(3)}(Y_1,Y_2,Y_3)
 =
27\sum_{(u,v,w):\,uvw\in E(C)}
a_1(u)a_2(v)a_3(w).
\]
Recall that $x=\int_0^1\one_{\{x\ge t\}}\dd t$ holds for each $x\in [0,1]$.
For $t\in[0,1]$, write $A_i(t):=\{v\in V(C):a_i(v)\ge t\}$.
For each $i\in[3]$,
\[
\int_0^1|A_i(t)|\,\dd t
 =\sum_{v\in V(C)}a_i(v)
 =\frac{|Y_i|}{3}.
\]
By Fubini's theorem and part~\ref{item:p2} of
\cref{lem:v-dominant},
\begin{align*}
 e_{C(3)}(Y_1,Y_2,Y_3)
 &=27\int_{[0,1]^3}
e_C\bigl(A_1(u),A_2(v),A_3(w)\bigr)
 \,\dd u \dd v \dd w\\
 &\ge
 \left(\rho^3-\frac{\eps}{2}\right)
 |Y_1||Y_2||Y_3|
 -\frac{\eps}{2}(3s)^3.
\end{align*}
Thus $C(3)$ is $(3s,\rho^3-\eps/2,\eps)$-dense and, by the choice of
$\eps$, satisfies the density condition in \cref{lem:perfect-matching}.
Moreover, for each $x\in V_v$, part~\ref{item:p3} of
\cref{lem:v-dominant} gives
\[
 \deg_{C(3)}(x)
 =9\deg_C(v)
 \ge
 9\left(\alpha+\rho-1-\frac{\eps}{2}\right)
 \binom{s-1}{2}
 \ge
 \alpha_{*}\binom{3s-1}{2},
\]
where the last inequality holds for all sufficiently large $s$. Since $3\mid3s$, \cref{lem:perfect-matching} yields a perfect matching
$M$ in $C(3)$.
For each $e\in E(C)$, let $m_e$ be the number of edges of $M$ whose cluster
type is $e$, and define $\omega(e):=m_e/3$.  Since $M$ is a matching,
$0\le m_e\le3$.  Moreover, $M$ covers all three vertices in each cluster,
so, for every $v\in V(C)$,
\[
 \sum_{e\ni v}\omega(e)
 =\frac13\sum_{e\ni v}m_e
 =1.
\]
Hence $\omega$ is a perfect fractional matching of $C$, and
\ref{item:F2} holds.

{\bf Verifying~\ref{item:F3}.}
By part~\ref{item:p2} of \cref{lem:v-dominant}, the component $C$ is
$(s,d_*,\mu_7)$-dense. Hence \cref{lem:C7} gives a copy of $C_7^{(3)}$
in $C$, and $7\equiv1\pmod3$.

{\bf Verifying~\ref{item:F4}.}
Let $J$ be an $(s+1)$-vertex $3$-graph and suppose that
$H=J-x,H'=J-y\in\cP_s$ for distinct vertices $x,y$. Put
$W:=V(J)\setminus\{x,y\}$. Since deleting one vertex removes at most
$s-1$ shadow edges, part~\ref{item:p1} of \cref{lem:v-dominant} gives
\[
 e_{\partial C(H)}(W),e_{\partial C(H')}(W)
 \ge(\rho-\eps)\binom s2-(s-1)
 >\frac12\binom{s-1}{2}
\]
for sufficiently large $s$. The two shadows therefore share a pair in
$\binom W2$, so $C(H)\cup C(H')$ is tightly connected in $J$.

Thus $\cP_s$ admits a Hamilton framework.
\end{proof}

\begin{lemma}[Codegree local Hamilton framework]
\label{lem:local-codeg-framework}
For every $d,\alpha>\kappa$, there exist $\mu>0$ and $s_0\in\N$ such
that, for every $s\ge s_0$, the family $\cP_s$ of all
$(s,d,\mu)$-dense $3$-graphs $H$ satisfying $\delta_2(H)\ge\alpha s$
admits a Hamilton framework.
\end{lemma}

\begin{proof}
Let $\tau>0$ be supplied by \cref{lem:codeg-dominant}. Choose
$q$ with $1/4<q<\alpha$, and apply
\cref{lem:shadow-consistency} with $\kappa+\tau$ and $q$, obtaining
$\eps_{\rm con}>0$. Set
$d_*=\alpha_*:=\tau/{2}$.
Apply \cref{lem:perfect-matching} with $d_*,\alpha_*$, obtaining
$\mu_{\rm pm}>0$. Choose
constants satisfying
\[
 0< 1/{s_0}\ll \mu\ll \eps<
 \min\left\{\frac{\eps_{\rm con}}4,\mu_{\rm pm},\frac{\tau}{2}\right\},
\]
such that $\alpha s-1\ge q(s-1)$ whenever $s\ge s_0$. Apply \cref{lem:codeg-dominant} with this error $\eps$.

For each $H\in\cP_s$, choose a component $C(H)$ supplied by
\cref{lem:codeg-dominant} and set $\cF(H):=C(H)$. Fix
$H\in\cP_s$, and write $C:=C(H)$.

Conditions~\ref{item:F1}--\ref{item:F3} follow as in the proof of
\cref{lem:local-vdeg-framework}. Indeed, $C$ is spanning, so
\ref{item:F1} holds. For \ref{item:F2}, form the balanced blow-up $C(3)$. The same calculation, together with the
density conclusion of \cref{lem:codeg-dominant}, shows that $C(3)$ is
$(3s,d_*,\mu_{\rm pm})$-dense. Moreover, for every clone $x$ of a
vertex $v\in V(C)$,
\[
 \deg_{C(3)}(x)
 =9\deg_C(v)
 \ge9\tau\binom{s-1}{2}
 \ge\alpha_*\binom{3s-1}{2}
\]
for sufficiently large $s$. Hence \cref{lem:perfect-matching} gives a
perfect matching in $C(3)$, whose projection to the cluster types gives
a perfect fractional matching in $C$. Finally, the copy of
$C_7^{(3)}$ supplied by \cref{lem:codeg-dominant} is a closed walk of
order $7\equiv1\pmod3$, proving \ref{item:F3}.

{\bf Verifying~\ref{item:F4}.}
Let $J$ be an $(s+1)$-vertex $3$-graph, and suppose that
$H:=J-x$ and $H':=J-y$ both belong to $\cP_s$ for distinct vertices
$x,y\in V(J)$. Set
\[
 C':=C(H')
 \qquad\text{and}\qquad
 W:=V(J)\setminus\{x,y\}.
\]
Suppose that $C$ and $C'$ are not tightly connected in $J$. Then the
graphs
\[
 R:=\partial C[W]
 \qquad\text{and}\qquad
 B:=\partial C'[W]
\]
are edge-disjoint, since a pair belonging to both shadows would be
contained in an edge of each component and would therefore tightly
connect them in $J$.
For all $X,Y\subseteq W$, the shadow conclusion of
\cref{lem:codeg-dominant} gives
\[
 e_R(X,Y),e_B(X,Y)
 \ge(\kappa+\tau)|X||Y|-\eps s^2\ge(\kappa+\tau)|X||Y|-\eps_{\rm con}|W|^2
\]
for sufficiently large $s$.
Now let $uv\in E(R)$. Since $uv\in\partial C$, every edge of $H$
containing $uv$ belongs to $C$. Consequently, every vertex in
$N_H(uv)\setminus\{y\}$ is a common neighbor of $u$ and $v$ in $R$.
Thus $uv$ has at least
\[
 \deg_H(uv)-1
 \ge\alpha s-1
 \ge q(s-1)
 =q|W|
\]
common neighbors in $R$. The same argument shows that every edge of
$B$ has at least $q|W|$ common neighbors in $B$. This contradicts
\cref{lem:shadow-consistency}. Hence $C\cup C'$ is tightly connected
in $J$, proving \ref{item:F4}.

Thus $\cP_s$ admits a Hamilton framework.
\end{proof}

\subsection{Proofs of the main theorems}\label{sec:main-proof}

\begin{proof}[Proof of \cref{thm:main-vexdeg-curve}]
Fix $d\in(1/3,1]$ and $\alpha>f(d)$. Choose
$\eps>0$ so small that
\[
 d_0:=d-\eps>\frac13,
 \qquad
 \alpha_0:=\alpha-\eps>f(d_0).
\]
Apply \cref{lem:local-vdeg-framework} with $d_0,\alpha_0$, and let
$\mu'>0$ and $s_0\in\N$ be the resulting constants. Choose
$s\ge s_0$ sufficiently large that
\cref{lem:density-inheritance,lem:degree-inheritance} apply with $r=6$
and target parameters $d_0,\mu'$ and $\ell=1,\alpha_0,\eps$, respectively,
and such that
$e^{-s^{1/6}}+e^{-\sqrt s}<1/s^2$.

Let $\cP_s$ be the family of all $(s,d_0,\mu')$-dense $3$-graphs $H$
with $\delta_1(H)\ge\alpha_0\binom{s-1}{2}$. This family admits a
Hamilton framework. Choose $\mu>0$ sufficiently small and then
$n_0\in\N$ sufficiently large for the two inheritance lemmas and
\cref{lem:HBT-special}.

Let $H$ be an $(n,d,\mu)$-dense $3$-graph with $n\ge n_0$ and
$\delta_1(H)\ge\alpha\binom{n-1}{2}$. Fix a $6$-set
$R\subseteq V(H)$. Since $H$ is also $(n,d_0,\mu)$-dense and
$\alpha=\alpha_0+\eps$, all but at most
\[
 \left(e^{-s^{1/6}}+e^{-\sqrt s}\right)\binom{n-6}{s-6}
\]
of the $s$-sets containing $R$ induce a member of $\cP_s$. Thus at least
$(1-1/s^2)\binom{n-6}{s-6}$ such sets do so, and
\cref{lem:HBT-special} gives a tight Hamilton cycle in $H$.
\end{proof}

\begin{proof}[Proof of \cref{thm:main-codeg}]
Choose $\eps>0$ so small that
\[
 d_0:=d-\eps>\kappa,
 \qquad
 \alpha_0:=\alpha-2\eps>\kappa.
\]
Apply \cref{lem:local-codeg-framework} with $d_0,\alpha_0$, obtaining
$\mu'>0$ and $s_0\in\N$. Choose $s\ge s_0$ sufficiently large that the
inheritance lemmas apply with $r=6$, density parameters $d_0,\mu'$, and
degree parameters $\ell=2,\alpha-\eps,\eps$, and such that
\[
 (\alpha-\eps)(s-2)\ge\alpha_0s,
 \qquad
 e^{-s^{1/6}}+e^{-\sqrt s}<\frac1{s^2}.
\]
Let $\cP_s$ be the family of all $(s,d_0,\mu')$-dense $3$-graphs $H$
with $\delta_2(H)\ge\alpha_0s$. This family admits a Hamilton framework.

Choose $\mu>0$ sufficiently small and then $n_0\in\N$ sufficiently large.
Let $H$ be an $(n,d,\mu)$-dense $3$-graph with $n\ge n_0$ and
$\delta_2(H)\ge\alpha n$. For every
$6$-set $R\subseteq V(H)$, all but at most
\[
 \left(e^{-s^{1/6}}+e^{-\sqrt s}\right)\binom{n-6}{s-6}
\]
of the $s$-sets containing $R$ induce an $(s,d_0,\mu')$-dense $3$-graph
with codegree at least $(\alpha-\eps)(s-2)\ge\alpha_0s$. Therefore at
least $(1-1/s^2)\binom{n-6}{s-6}$ such sets induce a member of $\cP_s$,
and \cref{lem:HBT-special} yields a tight Hamilton cycle in $H$.
\end{proof}

\section{Weighted edge-colored graphs}
\label{sec:weighted-edge-colored}
In this section, we study edge-colorings of graphs in which each
edge is assigned a vector of color weights.  Let $N,L\in\N$.  A
\emph{weighted $L$-coloring} of a graph $G$ assigns to every edge
$xy\in E(G)$ a vector
\[
 \mathbf w(xy)=\bigl(w_i(xy)\bigr)_{i\in[L]}\in[0,1]^L
 \text{~with~}
\sum_{i\in[L]}w_i(xy)\le1,
\]
and $\mathbf w(xy)$ is called  the \emph{weighted $L$-coloring vector} for $xy$.
We regard $w_i(xy)$ as the weight of color $i$ on the edge $xy$.
Equivalently, it may be viewed as the probability that $xy$ receives
color $i$, with the missing mass accounting for the possibility that no
color in $[L]$ is recorded.  If the coordinates sum to one on every
edge, we call $\mathbf w$ a \emph{perfect weighted $L$-coloring}.

The edge weights induce corresponding weights on monochromatic configurations. For distinct vertices $x,y,z\in V(G)$ with $xy, xz, yz \in E(G)$, define \[ \begin{aligned} M(x,y,z) &:=\sum_{i\in[L]}w_i(xy)w_i(xz)w_i(yz),\\ K_i(x,y) &:=\frac1N\sum_{z\in[N]\setminus\{x,y\}} w_i(xz)w_i(yz) \qquad (i\in[L]). \end{aligned} \] Here $M(x,y,z)$ is the total weight of the monochromatic triangle on $x,y,z$, while $K_i(x,y)$ is the normalized total weight of the $i$-color cherries with leaves $x,y$, where a \emph{cherry} is a path of length two.

This weighted model arises naturally from the tight-component coloring
of a $3$-graph.  Each pair in the shadow is colored by the tight
component containing it.  After applying regularity, an edge of the reduced graph records the densities of the different component
colors between the corresponding clusters and therefore carries a color
vector rather than a single color.  Under this correspondence,
monochromatic triangle weights record triangles whose three pairs lie in
a common component shadow, whereas the cherry weights retain the
common-neighbor information supplied by a codegree condition.  The
purpose of this section is to turn lower bounds on these local weights
into a color that dominates the reduced graph.  In the original
$3$-graph, such a color corresponds to a dominant tight component, which
provides the main structural input for the Hamilton-framework reduction.

\subsection{Dominant colors from heavy monochromatic triangles}
\label{sec:heavy-mono-clique}

We first consider weighted $L$-colorings for which the total
monochromatic weight $M(x,y,z)$ is large on every triangle.  The main
result of this subsection shows that, once this weight exceeds a threshold
$d\ge1/3$, a single color has weight greater than $d$ on every edge of the
complete graph.  We then derive a robust form in which the triangle
condition may fail on a small proportion of triples and the selected color
is heavy on almost every edge.  When an edge is denoted by $e$, we write
$w_i(e)$ for its color-$i$ weight.  We begin with two elementary
inequalities for weighted $L$-coloring vectors.

\begin{lemma}\label{lem:subprob-product}
Let $\mathbf a=(a_i)_{i\in[L]}$, $\mathbf b=(b_i)_{i\in[L]}$, and
$\mathbf c=(c_i)_{i\in[L]}$ be weighted $L$-coloring vectors, and let
$d\in[0,1]$. If
$\sum_{i\in[L]}a_i b_i c_i>d$, then some $i\in[L]$ satisfies
$a_i,b_i,c_i>d$.
\end{lemma}

\begin{proof}
Suppose not. Partition $[L]$ into $I_a,I_b,I_c$ by assigning each index
to one coordinate at which $\min\{a_i,b_i,c_i\}$ is attained. Since the
three coordinates cannot all exceed $d$,
\[
 \sum_{i\in[L]} a_i b_i c_i
 \le d\left(\sum_{i\in I_a}b_i c_i+
              \sum_{i\in I_b}a_i c_i+
              \sum_{i\in I_c}a_i b_i\right).
\]
Complete the three vectors to probability distributions by placing their
missing masses on three distinct auxiliary symbols, and let $A,B,C$ be
independent random variables with these distributions. The events
$\{B=C=i\}$ for $i\in I_a$, $\{A=C=i\}$ for $i\in I_b$, and
$\{A=B=i\}$ for $i\in I_c$ are pairwise disjoint. Their total probability
is the expression in parentheses and is therefore at most one, a
contradiction.
\end{proof}

\begin{lemma}\label{lem:pair-product}
Let $1/3\le d<1$ and $d<x,y,z\le1$. If
\[
 xyz+(1-x)(1-y)(1-z)>d,
\]
then $xy,xz,yz>d$ and
\[
 x+y+z>\frac{3+\sqrt{12d-3}}2.
\]
\end{lemma}

\begin{proof}
Suppose, without loss of generality, that $xy\le d$. For fixed $x,y$, the
left-hand side of the hypothesis is affine in $z$, so its maximum on
$[d,1]$ occurs at an endpoint. At $z=1$ it equals $xy\le d$. At $z=d$,
\[
 \frac{(1-x)(1-y)}{1-xy}
 \le\frac{(1-d)^2}{1-d^2}
 =\frac{1-d}{1+d}
 \le\frac d{1-d},
\]
where the last inequality uses $d\ge1/3$. Hence
$dxy+(1-d)(1-x)(1-y)\le d$, again a contradiction. Thus $xy>d$, and the
other two inequalities follow by symmetry.

Put $S:=x+y+z$. Since
$1-S+xy+xz+yz>d$ and $xy+xz+yz\le S^2/3$, we have
$S^2-3S+3-3d>0$. The roots of the corresponding quadratic are
$(3\pm\sqrt{12d-3})/2$. Since $S>3d$ and the smaller root is at most
$3d$, $S$ exceeds the larger root.
\end{proof}

\begin{lemma}[Weighted monochromatic clique]\label{thm:weighted-clique}
Let $N\ge4$, $L\ge1$, and $d\in[1/3,1)$, and let $\mathbf w$ be a
weighted $L$-coloring of $K_N$. Suppose that
$M(x,y,z)>d$ for every three distinct vertices $x,y,z\in[N]$. Then some
color $c\in[L]$ satisfies $w_c(e)>d$ for every $e\in E(K_N)$.
\end{lemma}

\begin{proof}
For each edge $e\in E(K_N)$, let 
$D_e:=\{i\in[L]:w_e(i)>d\}$ be the heavy-color set of $e$. Note that every $D_e$ has size at
most two. By \cref{lem:subprob-product}, every triangle $T$
has a color that is heavy on all three of its edges; we call any such color
a \emph{witness} for $T$.

Let $c$ be a witness for a triangle $T$ with edges $e_1,e_2,e_3$, and write
$x_j:=a_{e_j}(c)$ for $j\in[3]$. The contribution of the colors other than
$c$ is bounded by
\[
 \sum_{i\in[L]\setminus\{c\}}\prod_{j=1}^3a_{e_j}(i)
 \le\prod_{j=1}^3\sum_{i\in[L]\setminus\{c\}}a_{e_j}(i)
 \le\prod_{j=1}^3(1-x_j).
\]
Consequently,
\[
 x_1x_2x_3+(1-x_1)(1-x_2)(1-x_3)>d,
\]
and \cref{lem:pair-product} yields
\begin{equation}\label{eq:witness-pair-product}
 a_e(c)a_f(c)>d
 \qquad\text{for every two distinct edges }e,f\in E(T).
\end{equation}

We first prove the theorem for $K_4$. Label its vertices by $[4]$. For each
$j\in[4]$, let $F_j$ be the triangular face induced by
$[4]\setminus\{j\}$, and choose a witness color $c_j$ for $F_j$. For
$1\le j<k\le4$, let $e_{jk}$ denote the unique edge in
$E(F_j)\cap E(F_k)$. We shall repeatedly use the bound
\begin{equation}\label{eq:two-color-edge-bound}
 a_e(p)a_e(q)
 \le\frac{(a_e(p)+a_e(q))^2}{4}
 \le\frac14
 \qquad\text{whenever }p\ne q.
\end{equation}

If one color occurs at least three times among $c_1,c_2,c_3,c_4$, then every
edge of $K_4$ lies in at least one of the corresponding three faces and is
heavy in that color. We may therefore assume that no color occurs more than
twice. There are three possible multiplicity patterns.

\smallskip
\noindent\emph{Case 1: the pattern is $2+2$.}
Suppose that $c_1=c_2=a$ and $c_3=c_4=b$, where $a\ne b$. Applying
\eqref{eq:witness-pair-product} in the four faces gives
\[
 \begin{aligned}
 a_{e_{13}}(a)a_{e_{14}}(a)&>d,
 &\qquad a_{e_{23}}(a)a_{e_{24}}(a)&>d,\\
 a_{e_{13}}(b)a_{e_{23}}(b)&>d,
 & a_{e_{14}}(b)a_{e_{24}}(b)&>d.
 \end{aligned}
\]
Multiplying these inequalities and using
\eqref{eq:two-color-edge-bound} on each of the four displayed edges yields
\[
 d^4
 <\prod_{e\in\{e_{13},e_{14},e_{23},e_{24}\}}
   a_e(a)a_e(b)
 \le4^{-4},
\]
which is impossible because $d\ge1/3>1/4$.

\smallskip
\noindent\emph{Case 2: the pattern is $2+1+1$.}
Suppose that $c_1=c_2=a$, $c_3=b$, and $c_4=c$, where $a,b,c$ are
pairwise distinct. This time \eqref{eq:witness-pair-product} gives
\[
 \begin{aligned}
 a_{e_{13}}(a)a_{e_{14}}(a)&>d,
 &\qquad a_{e_{23}}(a)a_{e_{24}}(a)&>d,\\
 a_{e_{13}}(b)a_{e_{23}}(b)&>d,
 & a_{e_{14}}(c)a_{e_{24}}(c)&>d.
 \end{aligned}
\]
After multiplication, each of the four edges carries a product of two
distinct color weights. Hence \eqref{eq:two-color-edge-bound} again gives
$d^4<4^{-4}$, a contradiction.

\smallskip
\noindent\emph{Case 3: the pattern is $1+1+1+1$.}
Assume that $c_1,c_2,c_3,c_4$ are pairwise distinct. For each $j\in[4]$,
multiplying the three instances of \eqref{eq:witness-pair-product}
corresponding to the three pairs of edges in $F_j$ gives
\[
 \left(\prod_{e\in E(F_j)}a_e(c_j)\right)^2>d^3,
\]
and therefore
\[
 \prod_{e\in E(F_j)}a_e(c_j)>d^{3/2}.
\]
Multiplying over all four faces, and observing that the edge $e_{jk}$ is
contained precisely in $F_j$ and $F_k$, we obtain
\[
 d^6
 <\prod_{1\le j<k\le4}
   a_{e_{jk}}(c_j)a_{e_{jk}}(c_k)
 \le4^{-6}
\]
by \eqref{eq:two-color-edge-bound}. This is again impossible because
$d\ge1/3>1/4$. Thus the theorem holds for $K_4$.

We now proceed by induction on $m$. Let $m\ge5$ and assume that the result
holds for $K_{m-1}$. For every vertex $v\in V(K_m)$, the restriction to
$K_m-v$ satisfies the same triangle hypothesis, so there is a color $c_v$
that is heavy on every edge of $K_m-v$.

At most two distinct colors occur among the colors $c_v$. Indeed, if three
vertices $u,v,w$ had pairwise distinct colors $c_u,c_v,c_w$, then, since
$m\ge5$, there would be an edge disjoint from $\{u,v,w\}$. That edge would
be heavy in all three colors, contradicting the fact that an edge has at
most two heavy colors. Hence some color $c$ occurs as $c_v$ for at least
three vertices. Every edge of $K_m$ avoids at least one of these vertices
and is therefore heavy in color $c$. This completes the induction.
\end{proof}

For later applications, we need the following finite robust form, in which
a small exceptional set of triangles is allowed.

\begin{theorem}[Robust weighted clique selection]
\label{cor:robust-weighted-clique}
For every $t\in[1/3,1)$ and $\delta>0$, there exist $\xi>0$ and
$N_0\in\N$ such that the following holds for every $N\ge N_0$ and every
$L\in\N$. Let $\mathbf w$ be a weighted $L$-coloring of $K_N$. If all but
at most $\xi N^3$ triples $xyz\in\binom{[N]}3$ satisfy
$M(x,y,z)>t$, then some color $c\in[L]$ satisfies $w_c(e)>t$ on at least $(1-\delta)\binom N2$ edges.
\end{theorem}

\begin{proof}
We may assume that $0<\delta<1$, since the conclusion is trivial when
$\delta\ge1$. Choose an integer $s\ge4$ so large that $2e(1-\delta)^{s/2-1}<\frac12$.
Choose $\xi>0$ sufficiently small that $2\xi s^3<1/4$, and then choose
$N_0\ge s$ sufficiently large.

Call a triple \emph{exceptional} if its monochromatic weight is at most
$t$. A uniformly random $s$-set contains an exceptional triangle with
probability at most
\[
 \xi N^3\frac{\binom{N-3}{s-3}}{\binom Ns}
 \le2\xi s^3<\frac14.
\]
Thus at least three quarters of the $s$-sets contain no exceptional
triangle.

For every such $s$-set $S$, the restriction of $\mathbf w$ to
$K_N[S]$ satisfies
\[
 M(x,y,z)>t
 \qquad\text{for every }xyz\in\binom S3.
\]
By \cref{thm:weighted-clique}, there is a color $c=c(S)\in[L]$ such
that
\[
 w_c(e)>t
 \qquad\text{for every }e\in\binom S2.
\]

For each $i\in[L]$, let $G_i$ be the graph on $[N]$ whose edges are
precisely those $e$ satisfying $w_i(e)>t$, and let $q_i:=\frac{e(G_i)}{\binom N2}$.
The preceding paragraph shows that every nonexceptional $s$-set spans
a copy of $K_s$ in at least one of the graphs $G_i$. Hence
\begin{equation}\label{eq:robust-clique-cover}
 \frac34
 \le
 \sum_{i\in[L]}\frac{\#K_s(G_i)}{\binom Ns}.
\end{equation}

Since $t\ge1/3$ and the color vector on every edge is a subprobability
vector, every edge belongs to at most two of the graphs $G_i$.
Consequently,
\begin{equation}\label{eq:robust-color-mass}
 \sum_{i\in[L]}q_i\le2.
\end{equation}

For each $i\in[L]$, let $x_i\ge1$ be defined by
\[
 \binom{x_i}{2}=e(G_i).
\]
If $x_i<s$, then $G_i$ contains no copy of $K_s$. If $x_i\ge s$, the
Lov\'asz form of the Kruskal--Katona theorem gives
\[
 \#K_s(G_i)\le\binom{x_i}{s}.
\]
Moreover,
\[
 q_i=\frac{x_i(x_i-1)}{N(N-1)}.
\]
Since $s\le x_i\le N$,
\[
 \begin{split}
 \frac{\binom{x_i}{s}}{\binom Ns}
 &=
 \prod_{j=0}^{s-1}\frac{x_i-j}{N-j}
 \le\left(\frac{x_i}{N}\right)^s\\
 &\le
 \left(\frac{s}{s-1}\right)^{s/2}q_i^{s/2}
 <e\,q_i^{s/2}.
 \end{split}
\]

Suppose that $q_i\le1-\delta$ for every $i\in[L]$. Then
\[
 \begin{split}
 \sum_{i\in[L]}\frac{\#K_s(G_i)}{\binom Ns}
 &\le e\sum_{i\in[L]}q_i^{s/2}\\
 &\le e(1-\delta)^{s/2-1}\sum_{i\in[L]}q_i\\
 &\le2e(1-\delta)^{s/2-1}
 <\frac12,
 \end{split}
\]
contradicting \eqref{eq:robust-clique-cover}. Therefore
$q_i>1-\delta$ for some $i\in[L]$, which means that
$w_i(e)>t$ on more than $(1-\delta)\binom N2$ edges.
\end{proof}

We prove \cref{lem:codeg-dominant,lem:shadow-consistency}. Throughout this
section, $\kappa$ is the unique solution of $\kappa=(1-\kappa)^3$ in
$(1/4,1/3)$.

\subsection{Two colors from light monochromatic triangles}
\label{sec:light-mono-clique}

We first collect three auxiliary lemmas used in the proof of
\cref{lem:codeg-dominant}.

For $a,b,c\in[0,1]$, let
$\Phi(a,b,c):=abc+(1-a)(1-b)(1-c)$.
The following lemma shows a stability property of
$\Phi$ around $\kappa$.

\begin{lemma}
\label{lemma:kappa-inequality}
Let $\kappa<t<d<1$. For all $a,b,c\in[0,1]$, if
$a\le t$ and $\Phi(a,b,c)\ge d$, then $bc<\kappa$.
\end{lemma}

\begin{proof}
We prove the contrapositive. Suppose that $a\le t$ and
$bc\ge\kappa$. For fixed $b,c$, the function
$a\mapsto\Phi(a,b,c)$ is affine. Since $a\in[0,t]$, its maximum on
this interval is attained at $a=0$ or $a=t$.

Put $x:=\sqrt{bc}$. Then $x\in[\sqrt\kappa,1]$, and
$b+c\ge2x$ gives
\[
 (1-b)(1-c)=1-b-c+bc\le(1-x)^2.
\]
Recall that $\kappa>1/4$, and hence
$(1-\sqrt\kappa)^2<\kappa$.

At $a=0$, we therefore have
$\Phi(0,b,c)\le(1-\sqrt\kappa)^2<\kappa<t$. At $a=t$, since
$t<1$,
\[
 \Phi(t,b,c)\le g(x):=tx^2+(1-t)(1-x)^2.
\]
The function $g$ is convex, so its maximum on
$[\sqrt\kappa,1]$ is attained at an endpoint. Moreover,
$g(1)=t$, while
\[
 g(\sqrt\kappa)
 =t\kappa+(1-t)(1-\sqrt\kappa)^2
 \le t\kappa+(1-t)\kappa
 =\kappa<t.
\]
Thus $\Phi(t,b,c)\le t$.

Both endpoint values of the affine function
$a\mapsto\Phi(a,b,c)$ are therefore at most $t$, and hence
$\Phi(a,b,c)\le t<d$. Consequently, $\Phi(a,b,c)\ge d$ is
impossible when $a\le t$ and $bc\ge\kappa$, which proves that
$bc<\kappa$.
\end{proof}

We also use the following elementary set-system fact.

\begin{lemma}[Three-wise intersecting families of rank at most three]
\label{lem:rank-three}
Let $\mathcal A$ be a finite nonempty family of sets, each of size at
most three, such that
\[
 A\cap B\cap C\ne\varnothing
\]
for all $A,B,C\in\mathcal A$, where the three members are not required
to be distinct. Then either
$\bigcap_{A\in\mathcal A}A\ne\varnothing$, or there are four distinct
elements $c_1,c_2,c_3,c_4$ such that, writing
$Q:=\{c_1,c_2,c_3,c_4\}$, every member of $\mathcal A$ is of the form
$Q\setminus\{c_i\}$ for some $i\in[4]$, and all four possibilities occur.
\end{lemma}

\begin{proof}
Assume that $\bigcap_{A\in\mathcal A}A=\varnothing$, and choose an
inclusion-minimal subfamily $A_1,\ldots,A_r$ with empty intersection.
Since repetitions are allowed in the hypothesis, every subfamily of at
most three members has nonempty intersection, and hence $r\ge4$.

For each $j\in[r]$, minimality gives
$\bigcap_{k\ne j}A_k\ne\varnothing$. Since the intersection of all
$A_k$ is empty, we may choose
\[
 c_j\in\bigcap_{k\ne j}A_k\setminus A_j.
\]
These elements are distinct: if $i\ne j$, then $c_i\in A_j$, whereas
$c_j\notin A_j$. Moreover, for each $k$, the set $A_k$ contains every
$c_j$ with $j\ne k$. Thus $r-1\le |A_k|\le3$. Since $r\ge4$, it follows
that $r=4$ and
\[
 A_k=Q\setminus\{c_k\}
 \qquad\text{for every }k\in[4],
\]
where $Q=\{c_1,c_2,c_3,c_4\}$.

Now let $A\in\mathcal A$. For every distinct $i,j\in[4]$, the hypothesis
gives $A\cap A_i\cap A_j\ne\varnothing$. Since
$A_i\cap A_j=Q\setminus\{c_i,c_j\}$, these intersections run through all
two-element subsets of $Q$. Hence $A$ meets every two-element subset of
$Q$, and therefore contains at least three elements of $Q$. As
$|A|\le3$, we have $A=Q\setminus\{c_k\}$ for some $k\in[4]$. The four
possibilities already occur as $A_1,\ldots,A_4$, completing the proof.
\end{proof}

For a perfect weighted $L$-coloring $\mathbf w$ of $K_N$, define the
normalized color-$i$ weighted degree of a vertex $x$ and the normalized
total mass of color $i$ by
\[
 d_i(x):=\frac1N\sum_{y\in[N]\setminus\{x\}}w_i(xy),
 \qquad
 \overline w_i:=\frac1N\sum_{x\in[N]}d_i(x)
 =\frac{2}{N^2}\sum_{xy\in E(K_N)}w_i(xy).
\]
A color $i$ is \emph{present} if $\overline w_i>0$, or equivalently if
it has positive weight on at least one edge. We shall repeatedly use the
immediate bound
\begin{equation}\label{eq:Ki-degree-bound}
 K_i(x,y)\le \min\{d_i(x),d_i(y)\}.
\end{equation}

The next lemma turns approximate lower bounds on monochromatic triangle
weights and color-specific cherry weights into an almost-everywhere
dominant color.

\begin{lemma}[Robust dominant-color selection]
\label{lem:codeg-weighted-dominance}
Let $L\in\N$ and let $\kappa<t<d<1$. For every $\eps>0$, there exist
$\eta>0$ and $N_0\in\N$ such that the following holds. Let
$N\ge N_0$, $G$ be a graph on $[N]$, with $|E(G)| > (1- \eta) \binom{N}{2}$
and let $\mathbf w$ be a perfect weighted $L$-coloring of
$K_N$ satisfying:
\begin{enumerate}[label=\textup{(A\arabic*)},leftmargin=3.2em]
 \item\label{item:dominance-A1}
 $M(x,y,z)\ge d$ whenever $xy, xz, yz \in E(G)$;
 \item\label{item:dominance-A2}
 For every $xy\in E(G)$ and every $i\in[L]$, we have $K_i(x,y)\ge t$ whenever
 $w_i(xy)\ge\eta$.
\end{enumerate}
Then there exist a color $c\in[L]$ and a set
$\mathcal E\subseteq E(G)$ with $|\mathcal E|\le\eps N^2$ such that
every edge $xy\in E(G)\setminus\mathcal E$ satisfies:
\begin{enumerate}[label=\textup{(a\arabic*)},leftmargin=3.2em]
 \item\label{item:dominance-a1}
 $w_c(xy)>t$;
 
 \item\label{item:dominance-a2}
 for every $i\in[L]$, if $w_i(xy)\ge\eps$ then we have $w_i(xy)> t
  \ \text{ and }\ 
  K_i(x,y)\ge t$
\item\label{item:dominance-a3}
 \[
  \bigl|\{z\in[N]\setminus\{x,y\}:M(x,y,z)<d\}\bigr|
  \le\eps N.
 \]
\end{enumerate}
\end{lemma}

\begin{proof}
We first establish a light--heavy dichotomy on the edges of $G$ whose
endpoints have few nonneighbors. Since $t>\kappa$ and
$\kappa=(1-\kappa)^3$, we have $t>(1-t)^3$. As $d>t$, it follows that
$1-\sqrt{d/(1-t)}<t$. Choose $\zeta>0$ so small that
$\kappa+2\zeta<t$,
$2(1-\sqrt{d/(1-t)})+6\zeta<2t$, and
$4(L+2)\zeta<\eps$.

Let $U$ be the set of vertices having at most $\zeta N$ nonneighbors in
$G$. Since $G$ has fewer than $\eta\binom N2$ missing edges, we have
$|[N]\setminus U|\le \eta N/\zeta$. Moreover, if
$xy\in E(G[U])$, then at most $2\zeta N$ vertices $z$ fail to satisfy
$xz,yz\in E(G)$.

\smallskip
\noindent\emph{The light--heavy dichotomy.}
Fix $xy\in E(G[U])$ and $i\in[L]$, and suppose that
$\eta\le a:=w_i(xy)\le t$. For every $z$ with
$xz,yz\in E(G)$, put $b:=w_i(xz)$ and $c:=w_i(yz)$. By perfectness,
$M(x,y,z)\le \Phi(a,b,c)$, so \cref{lemma:kappa-inequality} and
\ref{item:dominance-A1} give $bc<\kappa$. The remaining choices of $z$
contribute at most $2\zeta$ to the normalized sum, and hence
$K_i(x,y)<\kappa+2\zeta<t$, contrary to
\ref{item:dominance-A2}. Thus every $xy\in E(G[U])$ satisfies
\begin{equation}\label{eq:dominance-quantization}
 w_i(xy)<\eta\quad\text{or}\quad w_i(xy)>t
 \qquad\text{for every }i\in[L].
\end{equation}
Call these two alternatives \emph{light} and \emph{heavy}, respectively.

Require $\eta$ to be small enough that
$L\eta<\min\{1,d-t\}$. Then every edge of $G[U]$ has a heavy color.
Every triangle of $G[U]$ also has a color that is heavy on all three
edges: otherwise each color contributes less than $\eta$ to its
monochromatic triangle weight, giving $M(x,y,z)<L\eta<d$.

\smallskip
\noindent\emph{A color active at almost every vertex.}
For $x\in[N]$, set $I(x):=\{i\in[L]:d_i(x)\ge t\}$ and
$A_i:=\{x\in[N]:i\in I(x)\}$. Since
$\sum_i d_i(x)=(N-1)/N<1$ and $t>\kappa>1/4$, every $I(x)$ has size at
most three. If $i$ is heavy on an edge $xy\in E(G[U])$, then
\ref{item:dominance-A2} and \eqref{eq:Ki-degree-bound} imply
$i\in I(x)\cap I(y)$.

Choose $\delta>0$ such that
$4\zeta<\delta$ and $(L+2)\delta<\eps$, and then choose an integer
$s\ge4$ with $3(1-\delta)^{s-1}<1/2$. By taking $\eta$ smaller if
necessary, a uniformly random $s$-set is contained in $U$ and spans a
clique of $G$ with probability at least $3/4$: indeed, the probability
of meeting $[N]\setminus U$ is at most $s\eta/\zeta$, while the
probability of containing a missing edge of $G$ is at most
$\eta\binom{s}{2}$.

Let $S$ be such an $s$-set. The family $\{I(x):x\in S\}$ is three-wise
intersecting. For three distinct vertices this follows from a common
heavy color of their triangle; the cases with repetitions follow from a
heavy color of the corresponding edge and from $I(x)\ne\varnothing$.

We claim that the exceptional configuration in \cref{lem:rank-three}
cannot occur. Otherwise, there are distinct vertices
$x_1,x_2,x_3,x_4\in S$ and distinct colors $c_1,c_2,c_3,c_4$ such that,
with $Q:=\{c_1,c_2,c_3,c_4\}$,
$I(x_j)=Q\setminus\{c_j\}$ for every $j\in[4]$. On the face opposite
$x_j$, the only color that can be heavy on all three edges is $c_j$.
Consequently, the product of the three $c_j$-weights on that face is at
least $d-L\eta$. Multiplying over the four faces and using
$w_i(e)w_j(e)\le1/4$ for distinct colors $i,j$ gives
\[
 (d-L\eta)^4\le4^{-6},
\]
which is impossible because $d-L\eta>t>1/4$.

Hence every such $S$ is contained in $A_i$ for some $i$. Since at least
three quarters of all $s$-sets have this property and
$\sum_i|A_i|\le3N$, if every $|A_i|<(1-\delta)N$, then
\[
 \frac34
 \le \sum_{i\in[L]}\frac{\binom{|A_i|}{s}}{\binom Ns}
 \le (1-\delta)^{s-1}\sum_{i\in[L]}\frac{|A_i|}{N}
 \le3(1-\delta)^{s-1}<\frac12,
\]
a contradiction. Therefore some color $c_0$ satisfies
$|A_{c_0}|\ge(1-\delta)N$.

\smallskip
\noindent\emph{From activity to dominance.}
Call an edge $xy\in E(G[U])$ \emph{$i$-pure} if
$w_i(xy)\ge1-L\eta$, and let $P_i$ be the set of $i$-pure edges.

Fix $a\in[L]$ and let $xy\in E(G[A_a\cap U])$ be an edge on which $a$
is light. Let $D$ be the set of heavy colors on $xy$ other than $a$.
This set is nonempty. For every $i\in D$, condition
\ref{item:dominance-A2} gives $K_i(x,y)\ge t$. Moreover,
$\sum_{i\ne a}K_i(x,y)\le1-d_a(x)-d_a(y)+K_a(x,y)$, and
\eqref{eq:Ki-degree-bound} therefore gives
\[
 |D|t\le \sum_{i\in D}K_i(x,y)
 \le1-d_a(x)-d_a(y)+K_a(x,y)\le1-t.
\]
Since $t>1/4$, we have $|D|\le2$.

Suppose that $D=\{i,j\}$. Then
$w_i(xy),w_j(xy)\le1-t$, while every other color is light on $xy$.
For every $z$ with $xz,yz\in E(G)$, condition
\ref{item:dominance-A1} gives
\[
 d\le(1-t)(1-w_a(xz))(1-w_a(yz))+L\eta.
\]
By the AM–GM inequality, averaging over $z$ yields
\[
 d_a(x)+d_a(y)
 \le2\left(1-\sqrt{\frac{d-L\eta}{1-t}}\right)
    +4\zeta+\frac2N<2t,
\]
provided that $\eta$ is sufficiently small and $N$ sufficiently large.
This contradicts $x,y\in A_a$. Thus $D$ consists of a single color $i$,
and \eqref{eq:dominance-quantization} shows that $xy$ is $i$-pure.

An $i$-pure edge propagates color $i$ to almost every vertex. Indeed, if
$xy$ is $i$-pure and $xz,yz\in E(G)$, then the total contribution of
colors other than $i$ to $M(x,y,z)$ is at most $L\eta$. Therefore
$w_i(xz)w_i(yz)\ge d-L\eta>t$. Hence $i$ is heavy on both $xz$ and
$yz$, and \ref{item:dominance-A2} together with
\eqref{eq:Ki-degree-bound} gives $i\in I(z)$. Consequently,
\begin{equation}\label{eq:dominance-pure-active}
 P_i\ne\varnothing\quad\Longrightarrow\quad
 |A_i|\ge(1-2\zeta)N.
\end{equation}

At most one color $i$ satisfies $|P_i|\ge\delta N^2$. Suppose that
$i\ne j$ both do. Fix $xy\in P_i$. Since $x,y\in U$, fewer than
$2\zeta N^2+2N<\delta N^2$ edges either meet $\{x,y\}$ or have an
endpoint nonadjacent in $G$ to $x$ or $y$. We may therefore choose
$uv\in P_j$ disjoint from $xy$ such that all four cross-pairs belong to
$G$. Applying the preceding product estimate to the four triangles
$xyu,xyv,uvx,uvy$ gives
\[
 (d-L\eta)^4
 \le\prod_{e\in\{xu,xv,yu,yv\}}w_i(e)w_j(e)
 \le4^{-4},
\]
again contradicting $d-L\eta>1/4$.

If some color has at least $\delta N^2$ pure edges, let $c$ be the
unique such color; otherwise let $c:=c_0$. In the first case,
\eqref{eq:dominance-pure-active} applies, and in the second we use the
choice of $c_0$. Since $2\zeta<\delta$, in either case
$|A_c|\ge(1-\delta)N$.

Every $c$-light edge in $G[A_c\cap U]$ is $i$-pure for some $i\ne c$.
By the choice of $c$, each such $P_i$ has fewer than $\delta N^2$ edges,
so there are fewer than $L\delta N^2$ $c$-light edges in
$G[A_c\cap U]$.

Now require $\eta/\zeta<\delta$ and $\eta<\eps$, and put
$W:=A_c\cap U$. Then $|[N]\setminus W|<2\delta N$. Let $\mathcal E$
consist of the edges of $G$ meeting $[N]\setminus W$ together with the
$c$-light edges of $G[W]$. We have
$|\mathcal E|<(L+2)\delta N^2<\eps N^2$.

Let $xy\in E(G)\setminus\mathcal E$. Then $x,y\in W$ and $c$ is not
light on $xy$, so \eqref{eq:dominance-quantization} gives
$w_c(xy)>t$, proving \ref{item:dominance-a1}. If
$w_i(xy)\ge\eps$, then $w_i(xy)\ge\eta$, and the same dichotomy gives
$w_i(xy)>t$; condition \ref{item:dominance-A2} then gives
$K_i(x,y)\ge t$, proving \ref{item:dominance-a2}.

Finally, if $M(x,y,z)<d$, then at least one of $xz,yz$ is missing from
$G$. Since $x,y\in U$, there are at most
$2\zeta N\le\eps N$ such vertices $z$, which proves
\ref{item:dominance-a3}.

Taking $\eta>0$ smaller than all bounds imposed above and then choosing
$N_0$ sufficiently large completes the proof.
\end{proof}

We now add a lower bound on the total mass of every present color.
This reduces the coloring to at most two colors and determines the
support of the possible second color.

\begin{theorem}[Structure of codegree-weighted colorings]
\label{thm:codeg-weighted-structure}
Let $\omega_0>0$ and let $\kappa<t<d<1$. For every $\eps>0$, there
exist $\eta>0$ and $N_0\in\N$ such that the following holds.

Let $N\ge N_0$, let $G$ be a graph on $[N]$ satisfying
\[
 e(G)>(1-\eta)\binom N2,
\]
and let $\mathbf w$ be a perfect weighted $L$-coloring of $K_N$, where
$L$ is arbitrary. Suppose that:
\begin{enumerate}[label=\textup{(C\arabic*)},leftmargin=3.2em]
 \item\label{item:structure-C1}
 $M(x,y,z)\ge d$ whenever $xy,xz,yz\in E(G)$;

 \item\label{item:structure-C2}
 for every $xy\in E(G)$ and every $i\in[L]$,
 \[
  w_i(xy)\ge\eta
  \quad\Longrightarrow\quad
  K_i(x,y)\ge t;
 \]

 \item\label{item:structure-C3}
 every present color satisfies $\overline w_i\ge\omega_0$.
\end{enumerate}
Then there exist a present color $c$, a set
$U\subseteq[N]$ with $|U|\ge(1-\eps)N$, and a graph $F$ on $U$ such
that
\[
 \overline G[U]\subseteq F
 \qquad\text{and}\qquad
 \Delta(F)\le\eps N,
\]
and the following properties hold:
\begin{enumerate}[label=\textup{(\roman*)},leftmargin=2.4em]
 \item\label{item:structure-concl1}
 for every $xy\in\binom U2\setminus E(F)$,
 \[
  w_c(xy)>t,
 \]
 and $d_c(x)\ge t-\eps$ for every $x\in U$;

 \item\label{item:structure-concl2}
 at most one present color is different from $c$;

 \item\label{item:structure-concl3}
 if $b\ne c$ is present and
 \[
  S:=\bigl\{x\in[N]:
  d_b(x)\ge(1-\kappa)^2\bigr\},
 \]
 then
 \[
  |S|\ge(1-\kappa-\eps)N,
 \]
 and, for every $xy\in\binom U2\setminus E(F)$,
 \[
  w_b(xy)>t
  \quad\Longleftrightarrow\quad
  x,y\in S.
 \]
 In particular, if at least one of $x,y$ lies outside $S$, then
 $w_b(xy)<\eps$.
\end{enumerate}
\end{theorem}

\begin{proof}
It is enough to prove the result for sufficiently small $\eps$, since
replacing $\eps$ by a smaller value only strengthens the conclusion.
Put
\[
 \lambda:=(1-\kappa)^2.
\]
The identity $\kappa=(1-\kappa)^3$ gives
\[
 \lambda=\frac{\kappa}{1-\kappa}
 \qquad\text{and}\qquad
 \frac{\kappa}{(1-\kappa)^2}=1-\kappa.
\]
Since $t>\kappa$, it follows that
\begin{equation}\label{eq:structure-fixed-margins}
 t>(1-t)\lambda,
 \qquad
 \frac{t}{(1-t)^2}>1-\kappa.
\end{equation}
Moreover, $\lambda>(1-t)/2$. We also have
$\lambda>1-\sqrt d$: indeed,
$d>\kappa$ and
$1-\sqrt\kappa<(1-\kappa)^2$.

Choose $\delta>0$ sufficiently small compared with $\eps,\omega_0$ and
the positive margins in \eqref{eq:structure-fixed-margins}, as well as
$t+2\lambda-1$ and $\lambda-1+\sqrt d$.

Since
\[
 \sum_{i\in[L]}\overline w_i=\frac{N-1}{N}<1,
\]
condition~\ref{item:structure-C3} implies that there are at most
$1/\omega_0$ present colors. Delete the zero colors and append zero
colors if necessary, so that the number of colors is bounded solely in
terms of $\omega_0$.

Apply \cref{lem:codeg-weighted-dominance} with output error
$\delta^2$. By decreasing $\eta$ and increasing $N_0$, we obtain a
present color $c$ and a set
$\mathcal E_0\subseteq E(G)$ with
\[
 |\mathcal E_0|\le\delta^2N^2
\]
such that every edge $xy\in E(G)\setminus\mathcal E_0$ satisfies
\begin{enumerate}[label=\textup{(\alph*)},leftmargin=2.4em]
 \item\label{item:structure-proof-a}
 $w_c(xy)>t$;

 \item\label{item:structure-proof-b}
 if $w_i(xy)\ge\delta^2$, then
 $w_i(xy)>t$ and $K_i(x,y)\ge t$;

 \item\label{item:structure-proof-c}
 at most $\delta^2N$ vertices $z$ satisfy
 $M(x,y,z)<d$.
\end{enumerate}

Let
\[
 B:=\overline G\cup\mathcal E_0
\]
and delete every vertex whose degree in $B$ exceeds $\delta N$. Denote
the remaining set by $U$ and put $F:=B[U]$. Since
$e(B)\le2\delta^2N^2$, we have
\[
 |[N]\setminus U|\le4\delta N.
\]
Furthermore,
\[
 \overline G[U]\subseteq F
 \qquad\text{and}\qquad
 \Delta(F)\le\delta N.
\]
Thus, by choosing $\delta$ sufficiently small,
\[
 |U|\ge(1-\eps)N
 \qquad\text{and}\qquad
 \Delta(F)\le\eps N.
\]

Every $x\in U$ is joined by an edge of
$\binom U2\setminus E(F)$ to all but $O(\delta N)$ vertices. By
\ref{item:structure-proof-a}, all these edges have color-$c$ weight
greater than $t$, and hence
\begin{equation}\label{eq:structure-dominant-degree}
 d_c(x)\ge t-O(\delta)\ge t-\eps
 \qquad\text{for every }x\in U.
\end{equation}
This proves \ref{item:structure-concl1}.

For each color $i\ne c$, define
\[
 S_i:=\{x\in[N]:d_i(x)\ge\lambda\}.
\]
We first show that every clean $i$-heavy edge has both endpoints in
$S_i$. Let
$xy\in\binom U2\setminus E(F)$ satisfy $w_i(xy)>t$. By
\ref{item:structure-proof-b}, $K_i(x,y)\ge t$. For all but
$O(\delta N)$ vertices $z$, the pair $yz$ lies in
$\binom U2\setminus E(F)$, and hence $w_c(yz)>t$. Therefore
$w_i(yz)\le1-t$, and
\[
 t\le K_i(x,y)\le(1-t)d_i(x)+O(\delta).
\]
By \eqref{eq:structure-fixed-margins} and the choice of $\delta$, this
gives $d_i(x)>\lambda$. The same argument gives
$d_i(y)>\lambda$. Thus
\begin{equation}\label{eq:structure-heavy-endpoints}
 w_i(xy)>t
 \quad\Longrightarrow\quad
 x,y\in S_i
\end{equation}
for every $xy\in\binom U2\setminus E(F)$.

If $i,j\ne c$ are distinct, then
\begin{equation}\label{eq:structure-support-disjoint}
 S_i\cap S_j\subseteq[N]\setminus U.
\end{equation}
Indeed, if $x\in U\cap S_i\cap S_j$, then
\eqref{eq:structure-dominant-degree} gives
\[
 1>\sum_kd_k(x)
 \ge d_c(x)+d_i(x)+d_j(x)
 \ge t-O(\delta)+2\lambda>1,
\]
a contradiction.

Now let $i\ne c$ be present. We claim that some edge
$xy\in\binom U2\setminus E(F)$ satisfies
$w_i(xy)\ge\delta^2$. Otherwise, the pairs meeting
$[N]\setminus U$, the edges of $F$, and the remaining pairs contribute
only $O(\delta)$ to $\overline w_i$, contradicting
$\overline w_i\ge\omega_0$. By
\ref{item:structure-proof-b}, such an edge is $i$-heavy and satisfies
$K_i(x,y)\ge t$.

For all but $O(\delta N)$ vertices $z\in S_i$, both $xz$ and $yz$ are
clean and $c$-heavy, so
\[
 w_i(xz)w_i(yz)\le(1-t)^2.
\]
If $z\notin S_i$ and $xz$ is clean, then
\eqref{eq:structure-heavy-endpoints} shows that $xz$ is not
$i$-heavy. Hence \ref{item:structure-proof-b} gives
$w_i(xz)<\delta^2$. Consequently,
\[
 t\le K_i(x,y)
 \le(1-t)^2\frac{|S_i|}{N}+O(\delta).
\]
Using \eqref{eq:structure-fixed-margins} and choosing $\delta$
sufficiently small, we obtain
\begin{equation}\label{eq:structure-large-support}
 |S_i|\ge(1-\kappa-\eps)N.
\end{equation}

There cannot be two distinct present colors $i,j\ne c$. Indeed,
\eqref{eq:structure-large-support} would give
\[
 |S_i\cap S_j|
 \ge |S_i|+|S_j|-N
 \ge(1-2\kappa-2\eps)N,
\]
whereas \eqref{eq:structure-support-disjoint} and
$|[N]\setminus U|=O(\delta N)$ give
$|S_i\cap S_j|=O(\delta N)$. Since $\kappa<1/3$, these bounds are
incompatible when $\eps$ and $\delta$ are sufficiently small. This
proves \ref{item:structure-concl2}.

Suppose finally that a second present color $b\ne c$ exists, and put
\[
 S:=S_b.
\]
The bound $|S|\ge(1-\kappa-\eps)N$ follows from
\eqref{eq:structure-large-support}. By
\eqref{eq:structure-heavy-endpoints}, every clean $b$-heavy edge has
both endpoints in $S$.

Conversely, let
$xy\in\binom U2\setminus E(F)$ with $x,y\in S$, and suppose that
$w_b(xy)\le t$. Then \ref{item:structure-proof-b} gives
\[
 u:=w_b(xy)<\delta^2.
\]
Since $c$ and $b$ are the only present colors, we have
$w_c(e)=1-w_b(e)$ on every edge $e$. By
\ref{item:structure-proof-c}, all but at most $\delta^2N$ vertices
$z$ satisfy $M(x,y,z)\ge d$. For each such $z$, writing
$p:=w_b(xz)$ and $q:=w_b(yz)$, we have
\[
 d\le(1-u)(1-p)(1-q)+upq.
\]
Since $u<\delta^2$, this implies
\[
 (1-p)(1-q)\ge d-2\delta^2.
\]
The arithmetic--geometric mean inequality therefore gives
\[
 p+q\le2\bigl(1-\sqrt{d-2\delta^2}\bigr).
\]
Averaging over $z$ and absorbing the exceptional vertices yields
\[
 d_b(x)+d_b(y)
 \le2\bigl(1-\sqrt{d-2\delta^2}\bigr)+O(\delta^2)
 <2\lambda,
\]
where the last inequality follows from
$\lambda>1-\sqrt d$ and the choice of $\delta$. This contradicts
$x,y\in S$. Hence every clean pair inside $S$ is $b$-heavy.

Together with \eqref{eq:structure-heavy-endpoints}, this proves
\[
 w_b(xy)>t
 \quad\Longleftrightarrow\quad
 x,y\in S
\]
for every $xy\in\binom U2\setminus E(F)$. If at least one endpoint lies
outside $S$, then $b$ is not heavy on $xy$, and
\ref{item:structure-proof-b} gives
$w_b(xy)<\delta^2<\eps$. This proves
\ref{item:structure-concl3}.
\end{proof}

We next prove a counting lemma for monochromatic cherries in a weighted
two-coloring of a complete graph.  As in the weighted
colorings considered in \cref{sec:pf-of-v-dominant}, we assign a color vector to each edge
of $K_N$.  Here the color set consists of red and blue, and each edge
$xy\in E(K_N)$ receives the probability vector
$\mathbf w_{xy}:=(w_{xy},1-w_{xy})$, where
$w_{xy}=w_{yx}\in[0,1]$.  Thus $w_{xy}$ and $1-w_{xy}$ are the red and
blue weights of the edge $xy$, respectively.

A cherry is a two-edge graph whose edges share a common endpoint.  More
precisely, for distinct vertices $x,y,z$, the edges $xz$ and $yz$ form a
cherry with center $z$ and leaves $x,y$.  If the colors of these two edges
are sampled independently according to their color vectors, then the
cherry is monochromatic red with probability $w_{xz}w_{yz}$ and
monochromatic blue with probability $(1-w_{xz})(1-w_{yz})$.
Accordingly, for each edge $xy\in E(K_N)$, define
\[
 \begin{aligned}
 K_W(x,y)
 &:=\frac1N\sum_{z\in[N]\setminus\{x,y\}}w_{xz}w_{yz},\\
 K_{1-W}(x,y)
 &:=\frac1N\sum_{z\in[N]\setminus\{x,y\}}
 (1-w_{xz})(1-w_{yz}).
 \end{aligned}
\]
Thus $K_W(x,y)$ and $K_{1-W}(x,y)$ are the normalized total weights of
the monochromatic red and blue cherries with leaves $x,y$, respectively.

The following lemma shows that, above the threshold $1/4$, the red and
blue cherry weights cannot both be large for almost every edge: for a
positive proportion of edges, at least one of the two monochromatic
cherry weights must be less than $q$.

\begin{lemma}[Weighted cherry dense]
\label{lemma:weighted cherry dense}
For every $q>1/4$ and every
$0<\beta<1-\frac1{4q}$, there exists $N_0\in\N$ such that, for every
$N\ge N_0$ and every weighted two-coloring on $K_N$, at least
$\beta N^2/2$ edges $xy\in E(K_N)$ satisfy
\[
 \min\{K_W(x,y),K_{1-W}(x,y)\}<q.
\]
\end{lemma}

\begin{proof}
Let
\[
 \mathcal B:=
 \{xy\in E(K_N):
   \min\{K_W(x,y),K_{1-W}(x,y)\}<q\}
\]
be the set of bad edges.  For each $v\in[N]$, let
$\mathcal G_v:=\{y\in[N]\setminus\{v\}:vy\notin\mathcal B\}$, so
$\mathcal G_v$ consists of the vertices joined to $v$ by good edges.
For every $y\in\mathcal G_v$, the total red and blue weights of the
cherries with leaves $v,y$ are both at least $qN$; that is,
\[
 \sum_{z\ne v,y}w_{vz}w_{yz}\ge qN
 \quad\text{and}\quad
 \sum_{z\ne v,y}(1-w_{vz})(1-w_{yz})\ge qN.
\]

Fix $v\in[N]$ and write $a_x:=w_{vx}$ for the red weight of the edge
$vx$.  For each $y\in\mathcal G_v$, multiply the red-cherry inequality
by the blue weight $1-a_y$ of $vy$, and multiply the blue-cherry
inequality by the red weight $a_y$ of $vy$.  Adding these inequalities,
summing over $y\in\mathcal G_v$, and interchanging $y$ and $z$ in the
second resulting sum gives
\[
 \begin{aligned}
 qN|\mathcal G_v|
 &\le
 \sum_{\substack{y\in\mathcal G_v\\ z\ne v,y}}
 a_z(1-a_y)w_{yz}
 +
 \sum_{\substack{z\in\mathcal G_v\\ y\ne v,z}}
 a_z(1-a_y)(1-w_{yz}).
 \end{aligned}
\]
For fixed distinct $y,z\ne v$, the factor $a_z(1-a_y)$ occurs with
coefficient $w_{yz}$ in the first sum only when $y\in\mathcal G_v$, and
with coefficient $1-w_{yz}$ in the second sum only when
$z\in\mathcal G_v$.  Its combined coefficient is therefore at most
$w_{yz}+(1-w_{yz})=1$.  Consequently,
\[
 \begin{aligned}
 qN|\mathcal G_v|
 &\le
 \sum_{\substack{y,z\ne v\\y\ne z}}a_z(1-a_y) =
 \left(\sum_{x\ne v}a_x\right)
 \left(N-1-\sum_{x\ne v}a_x\right)
 -\sum_{x\ne v}a_x(1-a_x)\\
 &\le
 \left(\sum_{x\ne v}a_x\right)
 \left(N-1-\sum_{x\ne v}a_x\right)
 \le \frac{(N-1)^2}{4}.
 \end{aligned}
\]
Thus
$|\mathcal G_v|\le (N-1)^2/(4qN)$ for every $v\in[N]$.

Every good edge $xy$ contributes once to $|\mathcal G_x|$ and once to
$|\mathcal G_y|$.  Hence
\[
 \begin{aligned}
 2\left(\binom N2-|\mathcal B|\right)
 &=\sum_{v=1}^N|\mathcal G_v|
 \le\frac{(N-1)^2}{4q},
 \end{aligned}
\]
and therefore
\[
 |\mathcal B|
 \ge
 \binom N2-\frac{(N-1)^2}{8q}
 =
 \left(1-\frac1{4q}+o(1)\right)\frac{N^2}{2}.
\]
Since $\beta<1-\frac1{4q}$, the right-hand side is at least
$\beta N^2/2$ for all sufficiently large $N$.
\end{proof}

\section{Proof of the main lemma in the vertex-degree case}
\label{sec:pf-of-v-dominant}

We now prove \cref{lem:v-dominant} using the Kruskal--Katona theorem, the
multicolor regularity lemma, and the weighted clique theorem above.

\begin{proof}
Fix $d\in(1/3,1]$, $\alpha>f(d)$, and $\eps>0$, and set
$\rho:=\rho(d)$. We may assume that $\eps<\alpha+\rho-1$.
Choose $d_0\in(1/3,d)$ sufficiently close to $d$ that, with $\rho_0:=\rho(d_0)$,
\begin{equation}\label{eq:vdeg-rho-margin}
\rho_0>\rho-\frac{\eps}{20},
\qquad
\rho_0^3>\rho^3-\frac{\eps}{20}.
\end{equation}

Apply \cref{cor:robust-weighted-clique} with threshold $d_0$ and error
$\eps/10^4$. Let $\xi>0$ and $m_0\in\N$ be the resulting constants.
Shrinking $\xi$ and increasing $m_0$ if necessary, assume that
\[
 \xi<\frac{\eps}{10^4},
 \qquad
 \frac1{m_0}<\frac{\eps}{10^4}.
\]
Choose
\[
 0<\eta<
 \min\left\{
  \frac{\xi}{100},
  \frac{\eps}{10^4},
  \frac{d-d_0}{100}
 \right\}.
\]

Apply \cref{lem:tail-triangles} with error $\eta\xi/100$. We obtain an
integer $L$ and an order threshold such that, whenever the component
shadows are listed in nonincreasing order of size,
\begin{equation}\label{eq:vdeg-tail-choice}
 \sum_{i>L}\#K_3(G_i)<\frac{\eta\xi}{200}n^3.
\end{equation}
Apply \cref{lem:multicolor-regularity} with $L$ colors, accuracy
$\eta/L$, and lower cluster bound $m_0$. Let $M_0$ be the resulting
upper cluster bound. Choose $\mu>0$ so that
\begin{equation}\label{eq:vdeg-mu-choice}
 \mu(2M_0)^3+7\eta<d-d_0,
 \qquad
 \mu<\frac{\eps}{4},
\end{equation}
and then choose $n_0$ larger than all preceding order thresholds and
sufficiently large for the estimates below.

Let $n\ge n_0$, and let $H$ be an $(n,d,\mu)$-dense $3$-graph satisfying
\[
 \delta_1(H)\ge\alpha\binom{n-1}{2}.
\]
List the tight components of $H$ as $C_1,C_2,\ldots$ in nonincreasing
order of their shadow sizes, and put
\[
 G_i:=\partial C_i.
\]
If necessary, append empty graphs so that $G_1,\ldots,G_L$ are defined.

Apply the multicolor regularity lemma to $G_1,\ldots,G_L$. We obtain a
partition
\[
 V(H)=V_0\cup V_1\cup\cdots\cup V_m,
 \qquad
 |V_1|=\cdots=|V_m|=:q,
\]
where
\[
 m_0\le m\le M_0,
 \qquad
 |V_0|\le\frac{\eta}{L}n.
\]
For $a\ne b$ and $i\in[L]$, put
\[
 w_i(ab):=d_{G_i}(V_a,V_b).
\]
Since the shadows of distinct tight components are edge-disjoint,
\begin{equation}\label{eq:vdeg-subprobability}
 \sum_{i=1}^L w_i(ab)\le1
 \qquad(a\ne b).
\end{equation}
Thus the vectors $(w_i(ab))_{i\in[L]}$ form a weighted $L$-coloring of
the reduced complete graph on $[m]$.

Call a cluster triple $abc\in\binom{[m]}3$ \emph{good} if its three
pairs are $\eta/L$-regular simultaneously in every $G_i$, $i\in[L]$,
and
\[
 \sum_{i>L}
 \#\bigl\{xyz\in K_3(G_i):
 x\in V_a,\ y\in V_b,\ z\in V_c\bigr\}
 \le\eta q^3.
\]
At most $(\eta/L)m^3$ cluster triples contain an irregular pair.
Moreover, if $b$ cluster triples violate the second condition, then
\[
 b\eta q^3
 \le\sum_{i>L}\#K_3(G_i)
 <\frac{\eta\xi}{200}n^3.
\]
Since $mq=n-|V_0|$ and $|V_0|\le\eta n/L$, we have
$n/q\le2m$, and hence
\[
 b<\frac{\xi}{200}\left(\frac nq\right)^3
 \le\frac{\xi}{25}m^3.
\]
By the choice of $\eta$, all but at most $\xi m^3$ cluster triples are
therefore good.

Fix a good triple $abc$. Uniform density gives
\[
 e_H(V_a,V_b,V_c)\ge dq^3-\mu n^3.
\]
Every edge counted here gives a monochromatic triangle in one component
shadow. The shadows with index greater than $L$ contribute at most
$\eta q^3$ such triangles. For each $i\in[L]$,
\cref{lem:regular-triangle} gives an error of at most
$6(\eta/L)q^3$. Summing over $i\in[L]$, we obtain
\[
 e_H(V_a,V_b,V_c)
 \le\bigl(M(a,b,c)+7\eta\bigr)q^3,
\]
where
\[
 M(a,b,c)
 =\sum_{i=1}^Lw_i(ab)w_i(ac)w_i(bc).
\]
Since $n/q\le2M_0$, \eqref{eq:vdeg-mu-choice} implies
\begin{equation}\label{eq:vdeg-reduced-M}
 M(a,b,c)
 \ge d-\mu\left(\frac nq\right)^3-7\eta
 >d_0.
\end{equation}

Thus all but at most $\xi m^3$ reduced triangles have monochromatic
weight greater than $d_0$. By
\cref{cor:robust-weighted-clique}, there is a color $c\in[L]$ such that
\begin{equation}\label{eq:vdeg-heavy-pairs}
 w_c(ab)>d_0
\end{equation}
for all but at most $(\eps/10^4)\binom m2$ cluster pairs $ab$.

Consider a good triple $abc$ whose three pairs satisfy
\eqref{eq:vdeg-heavy-pairs}, and write
\[
 x:=w_c(ab),\qquad y:=w_c(ac),\qquad z:=w_c(bc).
\]
By \eqref{eq:vdeg-subprobability}, the contribution of the remaining
tracked colors is at most $(1-x)(1-y)(1-z)$. Hence
\[
 d_0
 <M(a,b,c)
 \le xyz+(1-x)(1-y)(1-z).
\]
Put $S:=x+y+z$. Since
\[
 xyz+(1-x)(1-y)(1-z)
 =1-S+xy+xz+yz
 \le1-S+\frac{S^2}{3},
\]
we have
\[
 S^2-3S+3-3d_0>0.
\]
The roots of the corresponding quadratic are
$3(1-\rho_0)$ and $3\rho_0$. Since
$x,y,z>d_0>1-\rho_0$, we have $S>3(1-\rho_0)$, and therefore
\begin{equation}\label{eq:vdeg-triple-sum}
 w_c(ab)+w_c(ac)+w_c(bc)>3\rho_0.
\end{equation}

Equation~\eqref{eq:vdeg-triple-sum} can fail only if the cluster triple is
not good or contains a pair on which \eqref{eq:vdeg-heavy-pairs} fails.
The number of such triples is at most
\[
 \xi m^3+\frac{\eps}{10^4}\binom m2(m-2).
\]
Since every cluster pair lies in exactly $m-2$ cluster triples, summing
\eqref{eq:vdeg-triple-sum} gives
\[
 (m-2)\sum_{1\le a<b\le m}w_c(ab)
 \ge
 3\rho_0\left[
  \binom m3-\xi m^3
  -\frac{\eps}{10^4}\binom m2(m-2)
 \right].
\]
By the choices of $\xi$ and $m_0$,
\begin{equation}\label{eq:vdeg-weight-sum}
 \sum_{1\le a<b\le m}w_c(ab)
 \ge
 \left(\rho_0-\frac{\eps}{20}\right)\binom m2.
\end{equation}

The selected color is nonempty, so it corresponds to an actual tight
component. Put
\[
 C:=C_c,
 \qquad
 G:=\partial C.
\]
The cross-cluster edges counted in \eqref{eq:vdeg-weight-sum} give
\[
 \begin{split}
 e(G)
 &\ge q^2\sum_{a<b}w_c(ab)\\
 &\ge
 \left(\rho_0-\frac{\eps}{20}\right)
 q^2\binom m2.
 \end{split}
\]
Since
\[
 q^2\binom m2
 =\frac{(n-|V_0|)^2}{2}\left(1-\frac1m\right),
\]
the choices of $\eta,m_0,n_0$ and
\eqref{eq:vdeg-rho-margin} imply
\[
 e(G)\ge
 \left(\rho-\frac{\eps}{4}\right)\binom n2.
\]
In particular, part~\ref{item:p1} holds, and
\begin{equation}\label{eq:vdeg-complement-shadow}
 e(\overline G)
 \le
 \left(1-\rho+\frac{\eps}{4}\right)\binom n2.
\end{equation}

We next prove part~\ref{item:p2}. Every edge of $H\setminus C$ has all
three of its pairs in $\overline G$; otherwise it would be tightly
adjacent to an edge of $C$. For every reduced triple satisfying
\eqref{eq:vdeg-triple-sum}, the arithmetic--geometric mean inequality
gives
\begin{equation}\label{eq:vdeg-complement-product}
 \bigl(1-w_c(ab)\bigr)
 \bigl(1-w_c(ac)\bigr)
 \bigl(1-w_c(bc)\bigr)
 \le(1-\rho_0)^3.
\end{equation}

The complement of an $\eta/L$-regular pair is again
$\eta/L$-regular. Apply \cref{lem:regular-triangle} to $\overline G$
with the indicator functions of $X\cap V_a$, $Y\cap V_b$, and
$Z\cap V_c$, and sum over the ordered triples of distinct clusters for
which \eqref{eq:vdeg-triple-sum} holds. Tuples meeting $V_0$, tuples
using a cluster twice, exceptional reduced triples, and the regular
triangle-counting errors contribute at most
\[
 100\left(
 \eta+\xi+\frac{\eps}{10^4}+\frac1{m_0}
 \right)n^3
 \le\frac{\eps}{4}n^3.
\]
Consequently, uniformly for all $X,Y,Z\subseteq V(H)$,
\begin{equation}\label{eq:vdeg-outside-density}
 e_{H\setminus C}(X,Y,Z)
 \le
 (1-\rho_0)^3|X||Y||Z|+\frac{\eps}{4}n^3.
\end{equation}

Since $H$ is $(n,d,\mu)$-dense and
\[
 d_0=\rho_0^3+(1-\rho_0)^3,
\]
we obtain from \eqref{eq:vdeg-rho-margin},
\eqref{eq:vdeg-mu-choice}, and
\eqref{eq:vdeg-outside-density} that
\[
 \begin{split}
 e_C(X,Y,Z)
 &\ge
 \bigl(d-(1-\rho_0)^3\bigr)|X||Y||Z|
 -\left(\mu+\frac{\eps}{4}\right)n^3\\
 &=
 \bigl(\rho_0^3+d-d_0\bigr)|X||Y||Z|
 -\left(\mu+\frac{\eps}{4}\right)n^3\\
 &\ge
 (\rho^3-\eps)|X||Y||Z|-\eps n^3.
 \end{split}
\]
This proves part~\ref{item:p2}.

Finally, fix $v\in V(H)$. If
$vxy\in E(H)\setminus E(C)$, then $xy\notin E(G)$; otherwise
$vxy$ would be tightly adjacent to an edge of $C$. The map
$vxy\mapsto xy$ is injective, so
\[
 \deg_{H\setminus C}(v)\le e(\overline G).
\]
Using \eqref{eq:vdeg-complement-shadow}, we obtain, for sufficiently
large $n$,
\[
 \begin{split}
 \deg_C(v)
 &\ge
 \alpha\binom{n-1}{2}
 -\left(1-\rho+\frac{\eps}{4}\right)\binom n2\\
 &\ge
 (\alpha+\rho-1-\eps)\binom{n-1}{2}.
 \end{split}
\]
This proves part~\ref{item:p3}. Thus, the last quantity is positive for every
$v$, so every vertex lies in an edge of $C$. 
Hence $C$ is spanning.
\end{proof}

\section{Proofs of the main lemmas in the codegree case}
\label{sec:pf-of-co-dominant}

\subsection{The codegree dominant component}
\begin{proof}
We may assume that $d\le1$ and $\alpha<1$, since the remaining cases are
vacuous. Choose $\kappa<t<\min\{d,\alpha\}$ and put
$d_0:=(d+t)/2$. Since
$(1-t)^3<(1-\kappa)^3=\kappa<d$, we may fix $\tau>0$ such that
\begin{equation}\label{eq:codeg-tau-choice}
 \kappa+2\tau<t,\qquad
 (1-t)^3+2\tau<d,\qquad
 \tau<\frac{\alpha^2}{2}.
\end{equation}
This choice depends only on $d$ and $\alpha$.

Choose $0<\beta<1-1/(4t)$, and apply
\cref{lemma:weighted cherry dense} with $q=t$ and $\beta$.
Also apply \cref{lem:C7} with density $d$, which is possible because
$d>\kappa>4/27$; let $\mu_7>0$ and $n_7\in\N$ be the resulting
constants.

Fix $\eps>0$. Put $\gamma:=\beta/10$, and choose $\nu>0$ so that
$(2\nu)^{3/2}<\mu_7/2$. Choose $\xi>0$ sufficiently small in terms of
$\eps,\tau,\alpha-t,\beta,\nu\gamma$. Apply
\cref{thm:codeg-weighted-structure} with
$\omega_0=\alpha^2/2$, the thresholds $t,d_0$, and output error $\xi$.
Let $\eta_0>0$ and $r_0\in\N$ be the resulting constants. By decreasing
$\eta_0$ if necessary, we may assume that $\eta_0<t$.

Set $L:=\lceil\alpha^{-2}\rceil$. Choose the regularity accuracy
$\eta>0$ sufficiently small and the lower cluster bound $m_0$
sufficiently large so that $m_0$ exceeds both $r_0$ and the order
threshold in \cref{lemma:weighted cherry dense}, the error term in
\eqref{eq:finite-support-transfer}, with
$q=\alpha$, $\gamma=\eta_0$, and $\theta=2\sqrt\eta$, is smaller than
$\alpha-t$, and all losses of order $\sqrt\eta+1/m_0$ below are smaller
than the prescribed margins. Apply
\cref{lem:multicolor-regularity} with $L$ colors, accuracy $\eta$, and
lower bound $m_0$, and let $M$ be the resulting upper bound. Finally,
choose $\mu>0$ sufficiently small and then $n_0$ sufficiently large so
that
\[
 \mu(2M)^3+6L\eta<d-d_0,
 \qquad
 \mu(3/\gamma)^3<\frac{\mu_7}{2},
\]
and all errors below are at most the required fractions of $\eps$.

Let $H$ be an $(n,d,\mu)$-dense $3$-graph with $n\ge n_0$ and
$\delta_2(H)\ge\alpha n$. Since every pair has positive codegree, the
shadows of the tight components of $H$ partition $E(K_n)$. List the
tight components as $C_1,\ldots,C_h$ and put $G_i:=\partial C_i$. If
$xy\in E(G_i)$, then every edge of $H$ containing $xy$ belongs to
$C_i$. Thus $x$ and $y$ have at least $\alpha n$ common neighbors in
$G_i$. Consequently,
$3\#K_3(G_i)\ge\alpha n e(G_i)$, and
\cref{lem:kk-triangle} gives
\begin{equation}\label{eq:component-shadow-mass}
 e(G_i)\ge\frac{\alpha^2n^2}{2}.
\end{equation}
It follows that $h<\alpha^{-2}\le L$; append empty colors up to $L$.

Apply the regularity partition to $G_1,\ldots,G_L$, obtaining
\[
 V(H)=V_0\cup V_1\cup\cdots\cup V_r,
 \qquad |V_1|=\cdots=|V_r|=:m,
\]
where $m_0\le r\le M$ and $|V_0|\le\eta n$. Call an index
$a\in[r]$ bad if it is incident with more than $\sqrt\eta r$ cluster
pairs that are irregular in some color. There are at most
$2\sqrt\eta r$ bad indices. Let $\Gamma$ be the graph on $[r]$ whose
edges are the regular cluster pairs with two nonbad endpoints, and set
$w_i(ab):=d_{G_i}(V_a,V_b)$ for $a\ne b$, with $w_i(aa):=0$. Then
$e(\Gamma)>(1-\eta_0)\binom r2$, and
$\sum_iw_i(ab)=1$ for every $a\ne b$.

We verify the hypotheses of \cref{thm:codeg-weighted-structure}. If
$abc\in K_3(\Gamma)$, uniform density and
\cref{lem:regular-triangle} give
\[
 M(a,b,c)\ge d-\mu(n/m)^3-6L\eta\ge d_0.
\]
If $ab\in E(\Gamma)$ and $w_i(ab)\ge\eta_0$, then
\cref{lem:finite-support-transfer}, applied to $G_i$, gives
$K_i(a,b)\ge t$. Finally, \eqref{eq:component-shadow-mass} and the
negligible contribution of pairs meeting $V_0$ or lying inside one
cluster give
$\overline w_i\ge\alpha^2-O(\eta+1/m_0)\ge\alpha^2/2$
for every actual component color $i$. Thus all three hypotheses hold.

Apply \cref{thm:codeg-weighted-structure}. We obtain a present color
$c$, a set $U\subseteq[r]$ with $|U|\ge(1-\xi)r$, and a graph $F$ on
$U$ which contains every nonedge of $\Gamma[U]$ and satisfies
$\Delta(F)\le\xi r$. Call a pair in
$\binom U2\setminus E(F)$ clean. Every clean pair $ab$ satisfies
$w_c(ab)>t$. Moreover, there is at most one present color other than
$c$; if such a color $b$ exists and
\[
 S:=\{a\in[r]:d_b(a)\ge(1-\kappa)^2\},
\]
then, on every clean pair $ab$, we have
$w_b(ab)>t$ if and only if $a,b\in S$, while $w_b(ab)<\xi$ whenever
at least one endpoint lies outside $S$. Since every actual component
color has positive reduced mass, $H$ has at most two tight components.

Let $C:=C_c$. For a clean pair, \cref{lem:regular-cut} transfers the
inequality $w_c(ab)>t$ to arbitrary subsets of the corresponding
clusters. Summing over the reduced pairs and absorbing pairs meeting
$V_0$, pairs involving a cluster outside $U$, pairs in $F$, and
intracluster pairs into the error term gives, for all
$X,Y\subseteq V(H)$,
\[
 e_{\partial C}(X,Y)
 \ge t|X||Y|-\eps n^2
 \ge(\kappa+\tau)|X||Y|-\eps n^2.
\]
This proves part~\textup{(i)}.

Every edge of $H\setminus C$ has all three pairs outside $\partial C$.
On a reduced triple whose three pairs are clean, the complementary pair
densities are smaller than $1-t$. Hence
\cref{lem:regular-triangle}, summed over all reduced triples, gives
\[
 e_{H\setminus C}(X,Y,Z)
 \le(1-t)^3|X||Y||Z|+\frac{\eps}{2}n^3
\]
for all $X,Y,Z\subseteq V(H)$. Together with the density of $H$ and
\eqref{eq:codeg-tau-choice}, this yields
\[
 e_C(X,Y,Z)
 \ge\bigl(d-(1-t)^3\bigr)|X||Y||Z|-\eps n^3
 \ge\tau|X||Y||Z|-\eps n^3,
\]
proving part~\textup{(ii)}.

We next show that $C$ is spanning. Suppose not, and choose
$v\notin V(C)$. The unique other component, say $D=C_b$, contains
every edge of $H$ meeting $v$. Indeed, for each $u\ne v$, the pair
$uv$ has positive codegree and cannot lie in $\partial C$, so
$uv\in\partial D$. Every vertex in $N_H(uv)$ is then a common neighbor
of $u$ and $v$ in $\partial D$, and it follows that
$\delta(\partial D)\ge\alpha n$. Summing this degree bound over any
cluster $V_a$ gives
\begin{equation}\label{eq:second-reduced-degree}
 d_b(a)\ge\alpha-O(\sqrt\eta+1/m_0)>\frac{\alpha}{2}
 \qquad\text{for every }a\in[r].
\end{equation}
If $a\in U\setminus S$, then the conclusion of
\cref{thm:codeg-weighted-structure}, together with
$|[r]\setminus U|\le\xi r$ and $\Delta(F)\le\xi r$, gives
$d_b(a)\le3\xi$, contradicting \eqref{eq:second-reduced-degree}.
Thus $U\subseteq S$.

It follows that every clean pair is heavy in both colors $c$ and $b$.
Since clean pairs belong to $E(\Gamma)$ and $\eta_0<t$, the reduced
codegree condition gives $K_c(a,b),K_b(a,b)\ge t$ on every clean pair.
As $c$ and $b$ are the only present colors,
$w_b(ab)=1-w_c(ab)$ for all $a\ne b$. All but
$O(\xi r^2)<\beta r^2/2$ pairs are clean, contradicting
\cref{lemma:weighted cherry dense}. Hence $C$ is spanning.

Fix $v\in V(H)$ and choose $u$ with $uv\in\partial C$. Put
$A:=N_H(uv)=N_C(uv)$, so $|A|\ge\alpha n$. For every $z\in A$, the
pair $vz$ lies in $\partial C$, and hence
$\deg_C(vz)=\deg_H(vz)\ge\alpha n$. Each edge of $C$ containing $v$ is
counted at most twice in the following sum, so
\[
 2\deg_C(v)
 \ge\sum_{z\in A}\deg_C(vz)
 \ge\alpha^2n^2.
\]
By \eqref{eq:codeg-tau-choice}, this implies
$\delta_1(C)\ge\tau\binom{n-1}{2}$ for sufficiently large $n$, proving
part~\textup{(iii)}.

It remains to find $C_7^{(3)}$. If $C=H$, then the claim follows
directly from \cref{lem:C7}. Suppose that the second component
$D=C_b$ exists, and retain the set $S$ above. We claim that
\begin{equation}\label{eq:support-not-spanning}
 |S\cap U|\le(1-\gamma)r.
\end{equation}
Otherwise, the number of pairs that do not lie cleanly inside
$S\cap U$ is at most
$\gamma r^2+e(F)<\beta r^2/2$, while on every remaining pair both
$K_c$ and $K_b$ are at least $t$. This again contradicts
\cref{lemma:weighted cherry dense}.

Let $I:=U\setminus S$ and $T:=\bigcup_{a\in I}V_a$. By
\eqref{eq:support-not-spanning}, $|I|\ge\gamma r/2$ and
$|T|\ge\gamma n/3$. For every clean pair in $\binom I2$ we have
$w_b(ab)<\xi$. Therefore
\[
 e(\partial D[T])
 \le \xi\binom{|I|}{2}m^2
   +\frac{\xi r|I|m^2}{2}
   +\frac{|I|m^2}{2}
 \le\nu|T|^2,
\]
where the second term accounts for pairs in $F$ and the third for pairs
inside one cluster. By \cref{lem:kk-triangle},
\[
 e(D[T])\le\frac{(2\nu)^{3/2}}6|T|^3.
\]
Since $C$ and $D$ are the only tight components, for all
$X,Y,Z\subseteq T$ we obtain
\[
 e_{C[T]}(X,Y,Z)
 \ge d|X||Y||Z|
 -\left(\mu(3/\gamma)^3+(2\nu)^{3/2}\right)|T|^3
 \ge d|X||Y||Z|-\mu_7|T|^3.
\]
Thus $C[T]$ is $(|T|,d,\mu_7)$-dense, and $|T|\ge n_7$ by the choice
of $n_0$. Applying \cref{lem:C7} gives a copy of $C_7^{(3)}$ in
$C[T]$, proving part~\textup{(iv)}.
\end{proof}

\subsection{Consistency of overlapping components}

\begin{proof}[Proof of Lemma~\ref{lem:shadow-consistency}]
Suppose the assertion is false. We may assume that $t\le1/2$, since otherwise edge-disjointness and the two cut lower bounds with $X=Y=V(R)$ already give a contradiction for sufficiently small error. Then, for every $j$, there are edge-disjoint graphs $R_j,B_j$ on a common vertex set of order $s_j\to\infty$ such that
\[
 e_{R_j}(X,Y),e_{B_j}(X,Y)
 \ge t|X||Y|-\eps_js_j^2
\]
for all $X,Y$, where $\eps_j\to0$, and every edge of either graph has at least $qs_j$ common neighbors in that graph.

Choose a two-color regularity accuracy $\eta_j\to0$ and a lower cluster bound tending to infinity, and let $M_j$ be the corresponding upper bound. Since the negation supplies examples for arbitrarily small cut error, choose the sequence so that
\begin{equation}\label{eq:consistency-diagonal}
 \eps_jM_j^2\to0.
\end{equation}
Apply simultaneous regularity to $R_j$ and $B_j$, remove the reduced vertices of large irregular degree as in \cref{sec:preliminaries}, and write $r(xy)$ and $b(xy)$ for the reduced red and blue densities on the remaining $r_j\to\infty$ clusters. Applying the cut lower bounds with $X$ and $Y$ equal to two clusters, and using \eqref{eq:consistency-diagonal}, gives
\begin{equation}\label{eq:consistency-densities}
 r(xy),b(xy)\ge t-o(1)
\end{equation}
for every retained cluster pair. Edge-disjointness gives $r(xy)+b(xy)\le1$. Set
$r(xx):=b(xx):=0$ and put
\[
 K_R(x,y):=\frac1{r_j}\sum_z r(xz)r(yz),
 \qquad
 K_B(x,y):=\frac1{r_j}\sum_z b(xz)b(yz).
\]

Apply \cref{lem:finite-support-transfer} with the fixed threshold $\gamma=t/2$. For every regular cluster pair, equation \eqref{eq:consistency-densities} and the common-neighbor hypotheses give
\[
 K_R(x,y)\ge q-o(1),
 \qquad
 K_B(x,y)\ge q-o(1).
\]
Thus these inequalities hold for all but $o(r_j^2)$ pairs. Choose $q'$ with $1/4<q'<q$. For all sufficiently large $j$, both lower bounds exceed $q'$. Since $b(xy)\le1-r(xy)$, we also have $K_{1-R}(x,y)\ge K_B(x,y)$. The red weight system and its complement therefore have common-neighbor weights at least $q'$ on all but $o(r_j^2)$ pairs, contradicting.
\end{proof}

\section{Concluding remarks}\label{sec:remark}

In this paper, we determine the sharp minimum-vertex-degree
boundary for densities above $1/3$ and the sharp diagonal threshold in the
minimum-codegree setting. We conclude by formulating the remaining
questions in terms of the corresponding threshold functions.

For $d\in(0,1]$, let $h_1(d)$ be the infimum of all
$\alpha\in[0,1]$ such that, for every $\eps>0$, there exist $\mu>0$ and
$n_0\in\N$ for which every $(n,d,\mu)$-dense $3$-graph $H$ with
$n\ge n_0$ and
$\delta_1(H)\ge(\alpha+\eps)\binom{n-1}{2}$ contains a tight Hamilton
cycle. Define $h_2(d)$ analogously, replacing the degree condition by
$\delta_2(H)\ge(\alpha+\eps)n$.

For $d\in[1/4,1]$, recall that
$f(d)=(1-\sqrt{(4d-1)/3})/2$. By
\cref{thm:main-vexdeg-curve,prop:vertex-obstruction}, we have
$h_1(d)=f(d)$ for every $d>1/3$. 
Let $\kappa\in(1/4,1/3)$ be the unique solution of
$\kappa=(1-\kappa)^3$. Together with monotonicity of the density
condition, \cref{thm:codegree-counterexample,thm:main-codeg} imply
\[
h_2(d)\ge\kappa\quad(0<d\le\kappa),
\quad
h_2(d)\le\kappa\quad(d>\kappa).
\]
Thus $\kappa$ is the sharp diagonal codegree threshold. In particular, our result answers \cite[Problem~6.2]{Shu2026} affirmatively.

In independent recent work, Shu~\cite{Shu2026} determined the sharp
codegree curve above density $1/3$, showing that
$h_2(d)=f(d)^2$ for every $d>1/3$. Consequently,
$h_2(d)=h_1(d)^2$ throughout this range. 

For the remaining range, the currently known bounds are $\frac13\le h_1(d)\le\frac59$ when $0<d\le1/3$; $\kappa\le h_2(d)\le\frac12$ when $0<d\le\kappa$ and $f(d)^2\le h_2(d)\le\kappa$ when $\kappa<d\le1/3$.
Here the general upper bounds follow from
\cite{ReiherEtAl2019,RodlRucinskiSzemeredi2008}. The lower bounds
$h_1(d)\ge1/3$ and $h_2(d)\ge\kappa$ follow from the constructions in
\cref{prop:vertex-obstruction,prop:codegree-obstruction}, respectively,
together with monotonicity, while the last lower bound comes from the
construction in~\cite{Shu2026}.

The density $1/3$ also marks a limitation of the currently available
methods. The weighted monochromatic-clique theorem used in the
vertex-degree argument applies from this density onward, while the
arbitrary-pair connecting mechanism underlying the absorption approach
may fail below $1/3$~\cite{Shu2026}. It remains to understand whether the
unresolved ranges are governed by new extremal constructions or by
further changes in the structure of the relevant tight components.

\begin{problem}\label{prob:remaining-threshold-curves}
Determine $h_1(d)$ and $h_2(d)$ for $0<d\le1/3$. 
\end{problem}

\section*{Acknowledgements}
The authors acknowledge the employment of AI tools in the exploratory phase of this project. AI offered inspirations for lower bound constructions, and the authors developed the proof framework for upper bounds. All mathematical reasoning and proofs in the final manuscript were composed and verified by the authors.

\begingroup
\sloppy
\bibliographystyle{abbrv}
\bibliography{ref}
\endgroup

\end{document}